\documentclass[12pt]{amsart}
\usepackage{amsfonts,amssymb, amscd, latexsym, graphicx, psfrag, color, float}
\usepackage[all,2cell]{xy}
\UseTwocells
\usepackage{mathrsfs}
\usepackage{eucal}
\usepackage[colorlinks,
             linkcolor=black,
             citecolor=blue,
             bookmarksnumbered,hyperindex,hyperfigures,pdfpagelabels]{hyperref}
\usepackage{todonotes}
\usepackage{tikz-cd}

\newtheorem{dummy}{dummy}[section]
\newtheorem{lemma}[dummy]{Lemma}
\newtheorem{theorem}[dummy]{Theorem}
\newtheorem*{theorem*}{Theorem}

\newtheorem{corollary}[dummy]{Corollary}
\newtheorem{proposition}[dummy]{Proposition}
\theoremstyle{definition}
\newtheorem{definition}[dummy]{Definition}
\newtheorem{example}[dummy]{Example}
\newtheorem{remark}[dummy]{Remark}

\newcommand{\bA}{\mathbb{A}}
\newcommand{\bC}{\mathbb{C}}

\newcommand{\bG}{\mathbb{G}}

\newcommand{\bN}{\mathbb{N}}

\newcommand{\bR}{\mathbb{R}}
\newcommand{\bZ}{\mathbb{Z}}

\newcommand{\cC}{\mathcal{C}}
\newcommand{\cD}{\mathcal{D}}
\newcommand{\cE}{\mathcal{E}}

\newcommand{\cG}{\mathcal{G}}
\newcommand{\cH}{\mathcal{H}}
\newcommand{\cI}{\mathcal{I}}
\newcommand{\cJ}{\mathcal{J}}

\newcommand{\cO}{\mathcal{O}}

\newcommand{\cR}{\mathcal{R}}

\newcommand{\cS}{\mathcal{S}}
\newcommand{\cT}{\mathcal{T}}
\newcommand{\cU}{\mathcal{U}}
\newcommand{\cV}{\mathcal{V}}
\newcommand{\cX}{\mathcal{X}}
\newcommand{\cY}{\mathcal{Y}}

\newcommand{\sC}{\mathscr{C}}
\newcommand{\sD}{\mathscr{D}}
\newcommand{\sE}{\mathscr{E}}

\newcommand{\Aff}{\mathbf{Aff}^{\mathit{LFT}}_{\mathbb{C}}}
\newcommand{\Affk}{\mathbf{Aff}^{\mathit{LFT}}_{k}}

\newcommand{\As}{\mathbf{AlgSp}^{\mathit{LFT}}_{\mathbb{C}}}
\newcommand{\TopC}{\Top_\mathbb{C}}
\newcommand{\TopCs}{\TopC^s}

\newcommand{\Spec}{\mathrm{Spec}\,}

\newcommand{\Sh}{\operatorname{Sh}}

\newcommand{\Hom}{\mathrm{Hom}}

\newcommand{\Div}{\mathrm{Div}}

\renewcommand{\log}{{log}}
\newcommand{\topst}{{top}}
\newcommand{\gp}{\mathrm{gp}}
\newcommand{\an}{{an}}
\newcommand{\class}{\mathrm{B}}

\newcommand{\rank}{\mathrm{rank}}
\usepackage{xfrac}

\newcommand{\et}{\acute{e}t}

\newcommand*{\defeq}{\mathrel{\vcenter{\baselineskip0.5ex \lineskiplimit0pt
                     \hbox{\scriptsize.}\hbox{\scriptsize.}}}%
                     =}

\def\Top{\mathbf{Top}}
\def\Sch{\mathbf{Sch}}
\def\Algst{\mathbf{A}\!\mathfrak{lgSt}^{LFT}_{\mathbb{C}}}
\def\Topst{\mathfrak{TopSt}}

\renewcommand{\i}{\infty}
\def\iGpd{\operatorname{Gpd}_\i}
\def\LiGpd{\widehat{\operatorname{Gpd}}_\i}
\def\Pro{\operatorname{Pro}}
\def\Ind{\operatorname{Ind}}
\def\Fun{\operatorname{Fun}}
\DeclareMathOperator{\Psh}{Psh}
\def\Pshi{\operatorname{Psh}_\i}
\def\LPshi{\widehat{\operatorname{Psh}}_\i}
\def\Shi{\Sh_\i}
\def\Hshi{\mathbb{H}\mathrm{ypSh}_\i}
\def\colim{\underrightarrow{\mathrm{colim}\vspace{0.5pt}}\mspace{4mu}}
\newcommand{\hl}{\operatorname{holim}}

\newcommand{\holim}        {\underset{{\longleftarrow\!\!\!-\!\!\!-\!\!\!-\!\!\!-\!\!\!-\!\!\!-} } \hl  \:}
\renewcommand{\lim}{\varprojlim\mspace{3mu}}
\def\blank{\mspace{3mu}\cdot\mspace{3mu}}
\def\sfc{\mathcal{S}^{fc}}
\def\Profs{\operatorname{Prof}\left(\cS\right)}
\def\Prof{\operatorname{Prof}}

\def\Lan{\operatorname{Lan}}

\def\Top{\mathbf{Top}}
\def\Set{\mathit{Set}}

\def\Pip{\widehat{\Pi}_\i}

\def\longlongrightarrow{-\!\!\!-\!\!\!-\!\!\!-\!\!\!-\!\!\!-\!\!\!\longrightarrow}
\def\longlonglongrightarrow{-\!\!\!-\!\!\!-\!\!\!-\!\!\!-\!\!\!-\!\!\!\longlongrightarrow}

\newcommand*{\longlonghookrightarrow}{\ensuremath{\lhook\joinrel\relbar\joinrel\relbar\joinrel\rightarrow}}

\def\rrrarrow{\hspace{.05cm}\mbox{\,\put(0,-3){$\rightarrow$}\put(0,1){$\rightarrow$}\put(0,5){$\rightarrow$}\hspace{.45cm}}}

\begin{document}


\title[\resizebox{6in}{!}{KN-spaces, infinite root stacks, and the profinite homotopy type of log schemes}]
{Kato-Nakayama spaces, infinite root stacks, and the profinite homotopy type of log schemes}


\author{David Carchedi}
\address{Department of Mathematical Sciences\\
George Mason University\\
4400 University Drive, MS:  3F2\\
Exploratory Hall\\
Fairfax, Virginia  22030\\
USA}
\email{davidcarchedi@gmail.com}
\author{Sarah Scherotzke}
\address{Mathematisches Institut\\
Endenicher Allee 60\\
53115 Bonn\\
Germany}
\email{sarah@math.uni-bonn.de}
\author{Nicol\`o Sibilla}
\address{Department of Mathematics\\
University of British Columbia\\
1984 Mathematics Road\\
Vancouver, BC, V6T 1Z2\\
Canada}
\email{sibilla@math.ubc.ca}
\author{Mattia Talpo}
\address{Department of Mathematics\\
University of British Columbia\\
1984 Mathematics Road\\
Vancouver, BC, V6T 1Z2\\
Canada}
\email{mtalpo@math.ubc.ca}

\maketitle

\begin{abstract}
For a log scheme locally of finite type over $\mathbb{C}$, a natural candidate for its profinite homotopy type is the profinite completion of its Kato-Nakayama space \cite{KN}. Alternatively, one may consider the profinite homotopy type of the underlying topological stack of its infinite root stack \cite{TV}. Finally, for a log scheme not necessarily over $\mathbb{C}$, another natural candidate is the profinite \'etale homotopy type of its infinite root stack. We prove that, for a fine saturated log scheme locally of finite type over $\mathbb{C},$ these three notions agree. In particular, we construct a comparison map from the Kato-Nakayama space to the underlying topological stack of the infinite root stack, and prove that it induces an equivalence on profinite completions. In light of these results, we define the profinite homotopy type of a general fine saturated log scheme as the profinite \'{e}tale homotopy type of its infinite root stack.

\end{abstract}


\setcounter{tocdepth}{1}

\tableofcontents



\section{Introduction}

Log schemes are an enlargement of the category of schemes due to Fontaine, Illusie and Kato \cite{kato}. The resulting variant of algebraic geometry, ``logarithmic geometry'', has applications in a variety of contexts ranging from moduli theory to arithmetic and enumerative geometry (see \cite{abramovich} for a recent survey).

In the past years there have been several attempts to capture the ``log'' aspect of these objects and translate it into a more familiar terrain. In the complex analytic case, Kato and Nakayama introduced in \cite{KN} a topological space $X_\log$ (where $X$ is a log analytic space), which may be interpreted as the ``underlying topological space'' of $X$, and over which, in some cases, one can write a comparison between logarithmic de Rham cohomology and ordinary singular cohomology. In a different direction, for a log scheme $X,$ Kato 
introduced two sites, the Kummer-flat site $X_{Kfl}$ and the Kummer-\'{e}tale site $X_{Ket}$, that are analogous to the small fppf and \'{e}tale site of a scheme, and were used later by Hagihara and Nizio{\l} \cite{hagihara, Ni1} to study the K-theory of log schemes.

Recently in \cite{TV}, the fourth author together with Vistoli introduced and studied a third incarnation of the ``log aspect'' of a log structure, namely the \emph{infinite root stack} $\sqrt[\infty]{X}$, and used it to reinterpret Kato's Kummer sites and link them to parabolic sheaves on $X$. This stack is defined as the limit of an inverse system of algebraic stacks $\sqrt[\infty]{X}=\varprojlim_n \sqrt[n]{X}$, parameterizing $n$-th roots of the log structure of $X$.

The infinite root stack can be thought of as an ``algebraic incarnation'' of the Kato-Nakayama space: if $X$ is a log scheme locally of finite type over $\bC$, both $X_\log$ and $\sqrt[\infty]{X}$ have a map to $X$. The fiber of $X_\log\to X_\an$ over a point $x\in X_\an$ is homeomorphic to $(S^1)^r$, where $r$ is the rank of the log structure at $x$. For all $n,$ the reduced fiber of $\sqrt[n]{X}\to X$ over the corresponding closed point of $X$ is equivalent to the classifying stack $\class \left(\mathbb{Z}/{n\mathbb{Z}}\right)^r$ (for the same $r$). Regarding the infinite root stack not as the limit $\varprojlim_n \sqrt[n]{X},$ but instead as the diagram of stacks $$n \mapsto \sqrt[n]{X},$$ i.e. as a pro-object or ``formal limit,'' yields then that the reduced fiber of $\sqrt[\infty]{X}\to X$ is the diagram of stacks, $$n \mapsto \class \left(\mathbb{Z}/{n\mathbb{Z}}\right)^r,$$ which regarded as a pro-object is simply $\class \widehat{\mathbb{Z}}^r\simeq \widehat{\class{\mathbb{Z}}^r}$, the profinite completion of $\left(S^1\right)^r.$


In this paper we formalize this analogy and prove a comparison result between the profinite completions of $X_\log$ and $\sqrt[\infty]{X}$ for a fine saturated log scheme $X$ locally of finite type over $\bC$. Furthermore, we put this result in a wider circle of ideas, centered around the concept of the \emph{profinite homotopy type} of a log scheme.

Our approach relies in a crucial way on a careful reworking of the foundations of the theory of topological stacks and profinite completions within the framework of  $\infty$-categories \cite{htt}. This allows us to have greater technical control than earlier and more limited treatments, and  plays an important role in the proof of our main result.  
In the second half on the paper we construct a comparison map between $X_\log$ and $\sqrt[\infty]{X}$ and show that it is induces an equivalence between their profinite completions. The proof  involves an analysis of the local geometry of log schemes, and a local-to-global argument which reduces the statement to a local computation. Next, we review the main ideas in the paper in greater detail. 

\subsection{Topological stacks and profinite completions of homotopy types}
The first ingredient that we need in order to compare $X_\log$ and $\sqrt[\infty]{X}$ is the notion of a \emph{topological stack} \cite{No1} associated with an algebraic stack. This is an extension of the analytification functor defined on schemes and algebraic spaces, that equips algebraic stacks with a topological counterpart, and allows one for example to talk about their homotopy type. Given an algebraic stack $\mathcal{X}$ locally of finite type over $\bC$, let us denote by $\mathcal{X}_{top}$ its ``underlying topological stack''. This formalism allows us to carry over $\sqrt[\infty]{X}$ to the topological world, where $X_\log$ lives. 

The second ingredient we need is a functorial way of associating to a topological stack its homotopy type. Although this is in principle accomplished in \cite{No2} and \cite{No3}, the construction is a bit complicated and it is difficult to notice the nice formal properties this functor has from the construction. We instead construct a functor $\Pi_\i$ associating to a topological stack $\cX$ its \emph{fundamental $\i$-groupoid}. The source of this functor is a suitable $\i$-category of higher stacks on topological spaces, and the target is the $\i$-category $\cS$ of spaces. Using the language and machinery of $\i$-categories makes the construction and functoriality of $\Pi_\i$ entirely transparent; it is the unique colimit preserving functor which sends each space $T$ to its weak homotopy type.

The third ingredient we need is a way of associating to a space its profinite completion. Combining this with the functor $\Pi_\i$ gives a way of associating to a topological stack a profinite homotopy type. The notion of profinite completion of homotopy types is originally due to Artin and Mazur \cite{ArtinMazur}. Profinite homotopy types have since played many important roles in mathematics, perhaps most famously in relation to the Adams conjecture from algebraic topology \cite{Quillen,sullivan,Friedlander2}. A more modern exposition using model categories is given in \cite{Qu,Qu2,Isak1}; however the notion of profinite completion is a bit complicated in this framework. Finally, Lurie briefly introduces an $\i$-categorical model for profinite homotopy types in \cite{dagxiii}, which has recently been shown to be equivalent to Quick's model in \cite{prohomotopy} (and also to a special case of Isaksen's). The advantage of Lurie's framework is that the definition of profinite spaces and the notion of profinite completion become very simple. A \emph{$\pi$-finite space} is a space $X$ with finitely many connected components, and finitely many homotopy groups, all of whom are finite, and a \emph{profinite space} is simply a pro-object in the $\i$-category of $\pi$-finite spaces. The profinite completion functor $$\widehat{\left(\blank\right)}:\cS \to \Profs$$ from the $\i$-category of spaces to the $\i$-category of profinite spaces preserves colimits, and composing this functor with $\Pi_\i$ gives a colimit preserving functor $\Pip$ which assigns a topological stack its profinite homotopy type. This property is used in an essential way in the proof of our main theorem. Using this machinery, we are able to derive some non-trivial properties of profinite spaces that are used in a crucial way to prove our main result; in particular we show that profinite spaces can be glued along hypercovers (Lemma \ref{lem:hyperprof}).

\subsection{The comparison map and the equivalence of profinite completions}

Our main result states:

\begin{theorem*}[{see Theorem \ref{thm:main}}]\label{thrm:main}
Let $X$ be a fine saturated log scheme locally of finite type over $\bC$. 
Then there is a canonical map of pro-topological stacks 
$$
\Phi_X: X_{log} \rightarrow \sqrt[\infty]{X}_{top} 
$$
that induces an equivalence upon profinite completion
$$\Pip\left(X_{\log}\right)  \stackrel{\sim}{\longlongrightarrow} \Pip\left(\sqrt[\infty]{X}_{top} \right).$$
\end{theorem*}



This theorem makes precise the idea that the infinite root stack is an algebraic incarnation of the Kato-Nakayama space, and that it completely captures the ``profinite homotopy type'' (\`{a} la Artin-Mazur) of the corresponding log scheme.

The construction of the comparison map $\Phi_X$ is first performed \'{e}tale locally on $X$ where there is a global chart for the log structure, and then globalized by descent. The local construction uses the quotient stack description of the root stacks, that reduces the problem of finding a map to constructing a (topological) torsor on $X_\log$ with an equivariant map to a certain space. 

This permits the construction of $\Phi_X$ as a canonical morphism of pro-topological stacks over $X_{an}$: 
$$
\xymatrix{
X_{log} \ar[rr]^{\Phi_X} \ar[dr]_-{\pi_{log}} & & \sqrt[\infty]{X}_{top}  \ar[dl]^-{\pi_\infty} \\ 
& X_{an}. & }
$$
The jump patterns of the fibers of $\pi_{log}$ and $\pi_\infty$ reflect the way in which the rank of the 
log structure varies over $X_{an}$. More formally, the log structure  defines a canonical stratification on $X_{an}$ called the ``rank stratification'', which makes $X_{log}$ and $\sqrt[\infty]{X}_{top}$ into stratified fibrations. After profinite completion, 
the fibers of $\pi_{log}$ and $\pi_\infty$ on each stratum become equivalent: indeed they are equivalent 
respectively 
to real tori of dimension $n$, and to the (pro-)classifying stacks $\class \widehat{\mathbb{Z}}^n$. The fact that the fibers of $\pi_{log}$ and $\pi_\infty$ are profinite homotopy equivalent was in fact our initial intuition as to why the main result should be true. 
Extracting from this fiber-wise statement a proof that 
$\Phi_X$ induces an equivalence of profinite homotopy types requires a local-to-global argument  that makes full 
use of the $\infty$-categorical framework developed in the first half of the paper.


The Kato-Nakayama space models the topology of log schemes, but its applicability is limited to schemes over the complex numbers. Our results suggest 
 that  the infinite root stack encodes all the topological information of log schemes (or at least its profinite completion) in a way that is exempt from this limitation. More precisely, if $X$ is a log scheme locally of finite type over $\bC$, there are three natural candidates for its ``profinite homotopy type'': the profinite completion of the Kato-Nakayama space $X_\log$, the profinite \'{e}tale homotopy type of $\sqrt[\infty]{X}$ and the profinite completion of the (pro-)topological stack $\sqrt[\infty]{X}_\topst$. Theorem \ref{thm:main} and Theorem \ref{3hot} (proved in \cite{etalehomotopy}) 
imply that these three constructions give the same result. 
This justifies the definition of the profinite homotopy type for a log scheme $X$, even outside of the complex case, as the profinite \'{e}tale homotopy type of its infinite root stack $\sqrt[\infty]{X}$.

Another possible approach to this would be to define the homotopy type of a log scheme via Kato's Kummer-\'etale topos (see \cite{kato2}). As proved in \cite[Section 6.2]{TV}, this topos is equivalent to an appropriately defined small \'etale topos of the infinite root stack. It is not immediate, however, to link the resulting profinite homotopy time and the one that we define in the present paper. We plan to address this point in future work.
 

We believe that our results hold in the framework of \emph{log analytic spaces} as well. Even though root stacks of those have not been considered anywhere yet, the construction and results about them that we use in the present paper should carry through without difficulty, using some notion of ``analytic stacks'' instead of algebraic ones.

In recent unpublished work,  Howell and Vologodsky give a definition of the motive of a log schemes inside  Voevodsky's triangulated category of motives. Based on our results we 
expect that infinite root stacks should provide   
an alternative encoding of the motive of log schemes, or a profinite approximation of it. It is an interesting question to explore possible connections between these two viewpoints. 

\subsection*{Description of content}

The paper is structured as follows. 



In the first two sections we develop the framework necessary to associate profinite homotopy types to (pro-)algebraic and topological stacks. Along the way, in Section \ref{subsec.KN} we prove an interesting result (Theorem \ref{thm:KN homotopy}) which expresses the homotopy type of the Kato-Nakayama space of a log scheme as the classifying space of a natural category.



As a first step towards the main theorem, we construct in Section \ref{section: KN to IRS} (Proposition \ref{kntoirs}) a canonical map of pro-topological stacks 
\begin{equation}\label{eq:map}
\Phi_X\colon X_\log\to \sqrt[\infty]{X}_{\topst}
\end{equation}
by exploiting the local quotient stack presentations of the root stacks $\sqrt[n]{X}$, and gluing the resulting maps.

Section \ref{section:topology} contains results about the topology of the Kato-Nakayama space and the topological infinite root stack that we use in an essential way in the proof of our main result.

In Section \ref{section: the equivalence}, we give the proof of Theorem \ref{thm:main}:  we show that the canonical map (\ref{eq:map}) 
induces an equivalence after profinite completion. The proof is based on a local-to-global analysis: we use a suitable hypercover $U^\bullet$ of $X_{an}$ constructed in Section \ref{section:topology} to reduce the question to the restriction of the map $\Phi_X$ to each element of this hypercover. We then use the results about the topology of the Kato-Nakayama space and the topological infinite root stack proven in the same section to reduce to showing that the map induces a profinite homotopy equivalence along fibers. This concludes the proof. 

Finally, in Section \ref{section:prohottype} we make some remarks about the definition of the profinite homotopy type of a general log scheme.

In Appendix \ref{preliminaries}, we gather definitions and facts that we use throughout the paper about log schemes, the analytification functor, the Kato-Nakayama space, root stacks, and topological stacks. 
In particular, in (\ref{section: Stratification}), we carefully construct the ``rank stratification'' of $X$ (and $X_\an$), over which the characteristic monoid $\overline{M}$ of the log structure is locally constant.


\subsection*{Acknowledgements}

All of the authors would like to thank their respective home institutions for their support.

We are also happy to thank Kai Behrend, Thomas Goodwillie, Marc Hoyois, Jacob Lurie, Thomas Nikolaus, Behrang Noohi, Gereon Quick, Angelo Vistoli, and Kirsten Wickelgren for useful conversations. 

We are grateful to the anonymous referee for a careful reading and useful comments, in particular for pointing out the short proof of Proposition \ref{torsor}.


\subsection*{Notations and conventions}

We will always work over a field $k$, which will almost always be the complex numbers $\bC$. In particular all our log schemes will be fine and saturated, and locally of finite type over $\bC$, unless otherwise stated. 


If $P$ is a monoid we denote by $P^\gp$ the associated group. Our monoids will typically be integral, finitely generated, saturated and sharp (hence torsion-free). A monoid $P$ with these properties has a distinguished ``generating set'', consisting of all its indecomposable elements. This gives a presentation of any such monoid $P$ through generators and relations.


If $F$ is a sheaf of sets on the small \'{e}tale site of a scheme, its ``stalks'' will always be stalks on geometric points.

By an $\i$-category, we mean a quasicategory or inner-Kan complex. These are a model for $\left(\i,1\right)$-categories. We will follow very closely the notational conventions and terminology from \cite{htt}, and refer the reader to the index and notational index in op. cit. One slight deviation from the notational conventions just mentioned that will be made is that for $C$ and $D$ objects of an $\i$-category $\sC,$ we will denote by $\Hom_{\sC}\left(C,D\right)$ the space of morphisms from $C$ to $D$ in $\sC,$ rather than using the notation $\mbox{Map}_{\sC}\left(C,D\right),$ in order to highlight the analogy with classical category theory. A very brief heuristic introduction to $\i$-categories can be found in Appendix A of \cite{hafdave}. See also \cite{groth}.


\section{Profinite homotopy types}\label{section: profinite}

In this section we will introduce the $\i$-categorical model for profinite spaces that we will use in this article. This $\i$-category is introduced in \cite[Section 3.6]{dagxiii}; a profinite space will succinctly be a pro-object in the $\i$-category of $\pi$-finite spaces. This notion is equivalent to the notion of profinite space introduced by Quick in \cite{Qu,Qu2} (see \cite{prohomotopy}), but the machinery and language of $\i$-categories is much more convenient to work with. Most importantly, the notion of profinite completion becomes completely transparent in this set up, and it is left-adjoint to the canonical inclusion of profinite spaces into pro-spaces, and hence in particular preserves all colimits. We use this fact in an essential way in the proof of our main result, and we do not know how to prove the analogous fact about profinite completion in any other formalism.

We start first by reviewing the notion of ind-objects and pro-objects.

We will interchangeably use the notation $\cS$ and $\iGpd$ for the $\i$-category of spaces, and the $\i$-category of $\i$-groupoids. These two $\i$-categories are one and the same, and we will use the different notations solely to emphasize in what way we are viewing the objects.

Recall that for $\sD$ a small category, the category of ind-objects is essentially the category obtained from $\sD$ by freely adjoining formal filtered colimits. This construction carries over for $\i$-categories. Moreover, if $\sD$ is an essentially small $\i$-category, the $\i$-category of \textbf{ind-objects} in $\sD,$ $\Ind\left(\sD\right),$ admits a canonical functor $$j:\sD \to \Ind\left(\sD\right)$$satisfying the following universal property:\\

For every $\i$-category $\sE$ which admits small filtered colimits, composition with $j$ induces an equivalence of $\i$-categories

$$\Fun_{\mathit{filt.}}\left(\Ind\left(\sD\right),\sE\right) \to \Fun\left(\sD,\sE\right),$$
where $\Fun_{\mathit{filt.}}\left(\Ind\left(\sD\right), \sE \right)$ denotes the $\i$-category of all functors $\Ind\left(\sD\right) \to \sE$ which preserve filtered colimits.

A more concrete description of the $\i$-category $\Ind\left(\sD\right)$ is as follows. First, recall the following proposition:

\begin{proposition}[{\cite[Corollary 5.3.5.4]{htt}}]
Denote by $\Pshi\left(\sD\right)$ the $\i$-category of $\i$-presheaves on $\sD,$ that is, the functor category $$\Fun\left(\sD^{op},\iGpd\right).$$ Let $\sD$ be an essentially small $\i$-category and let $F:\sD^{op} \to \iGpd$ be an $\i$-presheaf. Then the following conditions are equivalent:
\begin{itemize}
\item[i)] The associated right fibration
$$\int_\sD F \to \sD$$
classified by $F$ has $\int_\sD F$ a filtered $\i$-category.\\
\item[ii)] There exists a small filtered $\i$-category $\cJ$ and a functor $$f:\cJ \to \sD$$ such that $F$ is the colimit of the composite
$$\cJ \stackrel{f}{\longrightarrow} \sD \stackrel{y}{\hookrightarrow} \Pshi\left(\sD\right)$$ (where $y$ denotes the Yoneda embedding).
\end{itemize}
and if $\sD$ has finite colimits, $i)$ and $ii)$ are equivalent to:
\begin{itemize}
\item[iii)] $F$ is left exact (i.e. preserves finite limits).
\end{itemize}
\end{proposition}
The $\i$-category $\Ind\left(\sD\right)$ may be described as the full subcategory of $\Pshi\left(\sD\right)$ satisfying the equivalent conditions $i)$ and $ii)$ (or $iii)$ if $\sD$ has finite colimits). In particular, this implies that $j$ is full and faithful, since it is a restriction of the Yoneda embedding. In a nutshell $\Pshi\left(\sD\right)$ is the $\i$-category obtained from $\sD$ by freely adjoining formal colimits, and $ii)$ above states that $\Ind\left(\sD\right)$ is the full subcategory thereof on those formal colimits of objects in $\sD$ which are filtered colimits.

The notion of a pro-object is dual to that of an ind-object; it is a formal cofiltered limit. By definition, the $\i$-category of \textbf{pro-objects} of an essentially small $\i$-category $\sD$ is $$\Pro\left(\sD\right)\defeq \Ind\left(\sD^{op}\right)^{op}.$$ If $\sD$ has small limits, we see that $\Pro\left(\sD\right)$ can be described as the full subcategory of $\Fun\left(\sD,\iGpd\right)^{op}$ on those functors $$F:\sD \to \iGpd$$ such that $F$ preserves finite limits. Since this definition makes sense even when $\sD$ is not essentially small, we make the following definition, due to Lurie:

\begin{definition}
If $\sE$ is any accessible $\i$-category with finite limits, then we define the $\i$-category of \textbf{pro-objects} of $\sE,$ $\Pro\left(\sE\right),$ to be the full subcategory of $\Fun\left(\sE,\iGpd\right)^{op}$ on those functors $F:\sE \to \iGpd$ which are accessible and preserve finite limits.
\end{definition}

\begin{remark}\label{rmk:univpro}
If $\sE$ is any accessible $\i$-category and $E$ is an object of $\sE,$ then the functor $$\Hom\left(E,\blank\right):\sE \to \iGpd$$ co-represented by $E$ is accessible and preserves all limits. This induces a fully faithful functor $$\sE \stackrel{j}{\hookrightarrow} \Pro\left(\sE\right).$$

The functor $j$ satisfies the following universal property:\\
If $\cD$ is any $\i$-category admitting small cofiltered limits, then composition with $j$ induces an equivalence of $\i$-categories

\begin{equation}\label{eq:pro}
\Fun_{\mathit{co-filt.}}\left(\Pro\left(\sE\right),\sD\right) \to \Fun\left(\sE,\sD\right),
\end{equation}
where $\Fun_{\mathit{co-filt.}}\left(\Pro\left(\sE\right),\sD\right)$ is the full subcategory of $\Fun\left(\Pro\left(\sE\right),\sD\right)$ spanned by those functors which preserve small cofiltered limits, see \cite[Proposition 3.1.6]{dagxiii}.
\end{remark}

\begin{remark}
If $\sC$ is any (not necessarily accessible) $\i$-category, there always exists an $\i$-category $\Pro\left(\sC\right)$ satisfying the universal property (\ref{eq:pro}). This is a special case of \cite[Proposition 5.3.6.2]{htt}.
\end{remark}

\begin{remark}\label{rmk:universe}
Let $\sE$ be any accessible $\i$-category which is not necessarily essentially small. Let $\cU$ be the Grothendieck universe of small sets and let $\cV$ be a Grothendieck universe such that $\cU \in \cV,$ so that we may regard $\cV$ as the Grothendieck universe of large sets. Let $\widehat{\operatorname{Gpd}}_\i$ denote the $\i$-category of $\i$-groupoids in the universe $\cV.$ By the proof of \cite[Proposition 3.1.6]{dagxiii}, it follows that the essential image of the composition
$$\Pro\left(\sE\right) \hookrightarrow \Fun\left(\sE,\iGpd\right)^{op} \hookrightarrow \Fun\left(\sE,\widehat{\operatorname{Gpd}}_\i\right)^{op}$$ consists of those functors $F:\sE \to \widehat{\operatorname{Gpd}}_\i$ for which there exists a small filtered $\i$-category $\cJ$ and a functor $$f:\cJ \to \sE^{op}$$ such that $F$ is the colimit of the composite
$$\cJ \stackrel{f}{\longrightarrow} \sE^{op} \hookrightarrow \Fun\left(\sE,\widehat{\operatorname{Gpd}}_\i\right).$$
\end{remark}

\begin{remark}\label{rmk:mappingpro}
In light of Remark \ref{rmk:universe}, any object $X$ of $\Pro\left(\sE\right),$ for $\sE$ an accessible $\i$-category, can be written as a cofiltered limit of a diagram of the form $$F:\cI \to \sE \stackrel{j}{\hookrightarrow} \Pro\left(\sE\right),$$ or in more informal notation, $$X = \underset{i \in \cI} \lim X_i.$$ Unwinding the definitions, we see that if $Y= \underset{j \in \cJ} \lim Y_j$ is another such object of $\Pro\left(\sE\right),$ then the usual formula for the morphism space holds:
$$\Hom_{\Pro\left(\sE\right)}\left(X,Y\right) \simeq \underset{j \in \cJ} \lim \underset{i \in \cI} \colim \Hom_{\sE}\left(X_i,Y_j\right).$$ 
Now suppose that $\sE$ has a terminal object $1.$ Then
$$\Hom_{\Pro\left(\sE\right)}\left(X,j\left(1\right)\right) \simeq \underset{i \in \cI} \colim \Hom_{\sE}\left(X_i,1\right).$$ Notice that each space $\Hom_{\sE}\left(X_i,1\right)$ is contractible since $1$ is terminal, and since $\left(-2\right)$-truncated objects (i.e. terminal objects) are closed under filtered colimits in $\cS$ by \cite[Corollary 5.5.7.4]{htt}, it follows that $\Hom_{\Pro\left(\sE\right)}\left(X,j\left(1\right)\right)$ itself is a contractible space, and hence we conclude that $j\left(1\right)$ is a terminal object.
\end{remark}

\begin{example}
Let $\sE=\cS$ be the $\i$-category of spaces. Then the $\i$-category of \textbf{pro-spaces}, $\Pro\left(\cS\right),$ can be identified with the opposite category of functors $F:\cS \to \cS$ such that $F$ is accessible and left exact. Notice that any space $X$ gives rise to a pro-space $$\Hom\left(X,\blank\right):\cS \to \cS$$ which moreover preserves all limits. Moreover if $F:\cS \to \cS$ is \emph{any} functor which preserves all limits, then by the Adjoint Functor theorem for $\i$-categories (\cite[Corollary 5.5.2.9]{htt}), $F$ must have a right adjoint $G$, and is moreover accessible by Proposition 5.4.7.7 of op. cit. This then implies that
\begin{eqnarray*}
\Hom\left(G\left(*\right),X\right) &\simeq& \Hom\left(*,F\left(X\right)\right)\\
&\simeq& F\left(X\right).
\end{eqnarray*}
Hence $F\simeq j\left(G\left(*\right)\right).$ We conclude that the essential image of $$j:\cS \hookrightarrow \Pro\left(\cS\right)$$ is precisely those $\i$-functors $\cS \to \cS$ which preserve all small limits.
\end{example}

\begin{proposition}
The functor $$T:\Pro\left(\cS\right) \stackrel{\Hom\left(j\left(*\right),\blank\right)}{\longlonglongrightarrow} \cS$$ is right adjoint to the canonical inclusion $j:\cS \to \Pro\left(\cS\right).$
\end{proposition}
\begin{proof}
By Remark \ref{rmk:universe}, $\Pro\left(\cS\right)^{op}$ may be identified with a subcategory of $\Fun\left(\cS,\widehat{\operatorname{Gpd}}_\i\right)$ of large $\i$-co-presheaves, and since limits commute with limits, this subcategory is stable under small limits. Note that this implies that $\Pro\left(\cS\right)$ is cocomplete. Since the Yoneda embedding into large $\i$-presheaves $$\cS^{op} \stackrel{y}{\hookrightarrow} \LPshi\left(\cS\right)$$ preserves small limits, it follows that $$j:\cS \hookrightarrow \Pro\left(\cS\right)$$ preserves small colimits. Since $\cS\simeq \Pshi\left(1\right),$ where $1$ is the terminal $\i$-category, and since $\Pro\left(\cS\right)$ is cocomplete, one has by \cite[Theorem 5.1.5.6]{htt} that $j\simeq \operatorname{Lan}_{y_1} \left(t\right)$ where $y_1$ is the Yoneda embedding $1 \to \cS$ and $t:1 \to \Pro\left(\cS\right)$ is the functor picking out the object $j\left(*\right).$ It follows immediately from the Yoneda lemma that $\Hom\left(j\left(*\right),\blank\right)$ is right adjoint to $\operatorname{Lan}_{y_1}\left(t\right).$
\end{proof}

\begin{remark}
Let $P:\cS \to \cS$ be a pro-space. By \cite[Proposition 5.4.6.6]{htt}, since $P$ is accessible, it follows that the associated left fibration $$\int_\cS P$$ is accessible, and hence has a small cofinal subcategory $$r:\cC_P \hookrightarrow \int_\cS P,$$ and $P$ may be identified with the limit of the composite $$\cC_P \stackrel{r}{\hookrightarrow} \int_\cS P \stackrel{\pi_P}{\longlongrightarrow} \cS \stackrel{j}{\hookrightarrow} \Pro\left(\cS\right).$$ We claim that $$T\left(P\right)\simeq \lim \pi_P \circ r.$$
Indeed:
\begin{eqnarray*}
T\left(P\right) &=& \Hom\left(j\left(*\right),P\right)\\
&\simeq& \Hom\left(j\left(*\right),\lim j \circ \pi_P \circ r\right)\\
&\simeq& \lim \Hom\left(j\left(*\right), j\circ \pi_P \circ r\right)\\
&\simeq& \lim \Hom\left(*,\pi_P \circ r\right)\\
&\simeq& \lim \pi_P \circ r.
\end{eqnarray*}
By the same proof, if one has $P$ presented as a cofiltered limit $P = \lim j\left(X_\alpha\right)$ of spaces, then $T\left(P\right) \simeq \lim X_\alpha.$ In fact, this holds more generally by the following proposition:
\end{remark}

\begin{proposition}\label{prop:limpro}
Let $\sC$ be an accessible $\i$-category which admits small filtered limits. Then the canonical inclusion $$j:\sC \hookrightarrow \Pro\left(\sC\right)$$ has a right adjoint $T$ and if $F:\cI \to \sC$ is a cofiltered diagram corresponding an object in $\Pro\left(\sC\right),$ then $T\left(F\right)= \lim F.$
\end{proposition}

\begin{proof}
By Remark \ref{rmk:universe}, composition with $$j:\sC \hookrightarrow \Pro\left(\sC\right)$$ induces an equivalence of $\i$-categories
$$\Fun_{\mathit{co-filt.}}\left(\Pro\left(\sC\right),\sC\right) \to \Fun\left(\sC,\sC\right),$$ so we can find a functor $T:\Pro\left(\sC\right) \to \sC$ and an equivalence $$\eta:id_\sC \stackrel{\sim}{\longrightarrow} T \circ j.$$ Let $Z$ be an arbitrary object of $\Pro\left(\sC\right),$ then we can write $Z = \underset{i \in \cI} \lim j\left(X_i\right).$ First note that since $\eta$ is an equivalence and $T$ preserves cofiltered limits (by definition), we have that for such a $Z,$ $$T\left(Z\right) \simeq \underset{i \in \cI} \lim X_i.$$ This shows that $T$ has the desired properties on pro-objects. Let us now show that $T$ is a right adjoint to $j$. Let $C$ be an object of $\sC,$ then we have
\begin{eqnarray*}
\Hom_\sC\left(D,T\left(Z\right)\right) &\simeq& \Hom_\sC\left(D,\underset{i \in \cI} \lim X_i\right)\\
&\simeq& \underset{i \in \cI} \lim \Hom_\sC\left(D,X_i\right),
\end{eqnarray*}
and since $j$ is fully faithful, we have for each $i$ $$\Hom_\sC\left(D,X_i\right)\simeq \Hom_{\Prof\left(\sC\right)}\left(j\left(D\right),j\left(X_i\right)\right).$$ It follows then that
\begin{eqnarray*}
\Hom_\sC\left(D,T\left(Z\right)\right) &\simeq& \underset{i \in \cI} \lim \Hom_{\Prof\left(\sC\right)}\left(j\left(D\right),j\left(X_i\right)\right)\\
&\simeq& \Hom_{\Prof\left(\sC\right)}\left(j\left(D\right),\underset{i \in \cI} \lim j\left(X_i\right)\right)\\
&=& \Hom_{\Prof\left(\sC\right)}\left(j\left(D\right),Z\right).
\end{eqnarray*}
\end{proof}

\begin{definition}
A space $X$ in $\cS$ is \textbf{$\pi$-finite} if all of its homotopy groups are finite, it has only finitely many non-trivial homotopy groups, and finitely many connected components.
\end{definition}

\begin{definition}
Let $\sfc$ denote the full subcategory of the $\i$-category $\cS$ on the $\pi$-finite spaces. $\sfc$ is essentially small and idempotent complete (and hence accessible). The $\i$-category of \textbf{profinite spaces} is defined to be the $\i$-category $$\Profs \defeq \Pro\left(\sfc\right).$$
\end{definition}

\begin{proposition}\label{prop:pitruncated}
Let $V$ be a $\pi$-finite space. Note that $V$ is $n$-truncated for some $n,$ since it has only finitely many homotopy groups. The associated profinite space $j\left(V\right)$ is also $n$-truncated.
\end{proposition}

\begin{proof}
Let $X=\underset{i \in \cI} \lim X_i$ be a profinite space. Then by Remark \ref{rmk:mappingpro}, we have that
$$\Hom_{\Profs}\left(X,j\left(V\right)\right) \simeq \underset{i \in \cI} \colim \Hom_{\cS^{fc}}\left(X_i,V\right).$$ Each space $\Hom_{\cS^{fc}}\left(X_i,V\right)$ is $n$-truncated since $V$ is, and $n$-truncated spaces are stable under filtered colimits by \cite[Corollary 5.5.7.4]{htt}, so it follows that  $\Hom_{\Profs}\left(X,j\left(V\right)\right)$ is also $n$-truncated.
\end{proof}

\begin{remark}\label{rmk:fun}
The assignment $\sC \mapsto \Pro\left(\sC\right)$ is functorial among accessible $\i$-categories with finite limits. Given a functor $f:\sC \to \sD,$ the composite $$\sC \stackrel{f}{\longrightarrow} \sD \stackrel{j}{\hookrightarrow}\Pro\left(\sD\right)$$ corresponds to an object of the $\i$-category $\Fun\left(\sC,\Pro\left(\sD\right)\right),$ which by Remark \ref{rmk:univpro} is equivalent to the $\i$-category $\Fun_{\mathit{co-filt.}}\left(\Pro\left(\sC\right),\Pro\left(\sD\right)\right).$ Hence, one gets an induced functor $$\Pro\left(f\right):\Pro\left(\sC\right) \to \Pro\left(\sD\right)$$ which preserves cofiltered limits. Moreover, $\Pro\left(f\right)$ is fully faithful if $f$ is. If $f$ happens to be accessible and left exact, then there is an induced functor in the opposite direction, given by
\begin{eqnarray*}
f^*:\Pro\left(\sD\right) &\to& \Pro\left(\sC\right)\\
\left(\sD \stackrel{F}{\longrightarrow} \iGpd\right) &\mapsto& \left(\sC \stackrel{f}{\longrightarrow}\sD \stackrel{F}{\longrightarrow} \iGpd\right),
\end{eqnarray*}
and $f^*$ is left adjoint to $\Pro\left(f\right).$ See \cite[Remark 3.1.7]{dagxiii} (but note there is a typo, since $f^*$ is in fact a left adjoint, not a right adjoint).
\end{remark}

\begin{example}
The canonical inclusion $i:\sfc \hookrightarrow \cS$ induces a fully faithful embedding 
$$\Pro\left(i\right):\Prof\left(\cS\right) \hookrightarrow \Pro\left(\cS\right)$$
of profinite spaces into pro-spaces. Moreover, $i$ is accessible and preserves finite limits, hence the above functor has a left adjoint $$i^*:\Pro\left(\cS\right) \to \Prof\left(\cS\right).$$ This functor sends a pro-space $P$ to its \textbf{profinite completion}.
\end{example}

\begin{definition}
We denote by $\widehat{\left(\blank\right)}$ the composite
$$\cS \stackrel{j}{\hookrightarrow} \Pro\left(\cS\right) \stackrel{i^*}{\longrightarrow} \Prof\left(\cS\right)$$ and call it the \textbf{profinite completion functor}. Concretely, if
$X$ is a space in $\cS,$ then $\widehat{X}$ corresponds to the composite $$\sfc \stackrel{i}{\hookrightarrow} \cS \stackrel{\Hom\left(X,\blank\right)}{\longlonglongrightarrow} \cS.$$ This functor has a right adjoint given by the composite $$\Prof\left(\cS\right) \stackrel{\Pro\left(i\right)}{\longlonghookrightarrow} \Pro\left(\cS\right) \stackrel{T}{\longrightarrow} \cS.$$ We will denote this right adjoint simply by $U$.
\end{definition}

\begin{remark}
We will sometimes abuse notation and denote the profinite completion of a pro-space $Y$ also by $\widehat{Y}$ rather than $i^*Y,$ when no confusion will arise.
\end{remark}

\subsection{The relationship with profinite groups}

In this subsection, we will touch briefly upon the relationship between profinite groups and profinite spaces. Recall the notion of profinite completion of a group. A profinite group is a pro-object in the category of finite groups. Equivalently, a profinite group is a group object in profinite sets, see \cite[Proposition 3.2.12]{dagxiii}. 



Denote by $i:FinGp \hookrightarrow Gp$ the fully faithful inclusion of the category of finite groups into the category of groups. The composite
$$Gp \hookrightarrow \Pro\left(Gp\right) \stackrel{i^*}{\longlongrightarrow} \Pro\left(FinGp\right)\simeq Gp\left(\Pro\left(FinSet\right)\right)$$ is the functor assigning a group its profinite completion. We also denote this functor by $\widehat{\left(\blank\right)}$ when no confusion will arise. Recall that the profinite completion of a group has a classical concrete description as follows: Let $G$ be a group, then its profinite completion is the limit $\underset{N} \lim j\left(G/N\right),$ where $N$ ranges over all the finite index normal subgroups of $G.$
Similarly, denote by $i_{ab}:FinAbGp \hookrightarrow AbGp$ the fully faithful inclusion of the category of finite abelian groups into the category of abelian groups. By the analogous construction to the above, there is an induced profinite completion functor
$$\widehat{\left(\blank\right)}_{ab}:AbGp \to \Pro\left(FinAbGp\right).$$ It can be described classically by the same formula as in the non-abelian case. If $$\phi:AbGp \hookrightarrow Gp$$ is the canonical inclusion of abelian groups into groups, it follows that the following diagram commutes up to canonical natural equivalence:
$$\xymatrix@C=2.5cm{AbGp \ar[r]^-{\widehat{\left(\blank\right)}_{ab}} \ar[d]_-{\phi} & \Pro\left(FinAbGp\right) \ar[d]^-{\Pro\left(\phi\right)}\\Gp \ar[r]^-{\widehat{\left(\blank\right)}} & \Pro\left(FinGp\right).}$$ By \cite[Proposition 3.2.14]{dagxiii}, there is a canonical equivalence of categories $$\Pro\left(FinAbGp\right) \simeq AbGp\left(\Pro\left(FinSet\right)\right),$$ between the category of pro-objects in finite abelian groups and the category of abelian group objects in profinite sets. In particular, finite coproducts (direct sums) in $\Pro\left(FinAbGp\right)$ coincide with finite products. Since $\widehat{\left(\blank\right)}_{ab}$ is a left adjoint, it preserves direct sums, and by Remark \ref{rmk:fun}, $\Pro\left(\phi\right)$ is a right adjoint (since $\phi$ preserves finite limits), so $\Pro\left(\phi\right)$ preserves products. It follows that the composite $$\widehat{\left(\blank\right)} \circ \phi:AbGp \to \Pro\left(FinGp\right)$$ preserves finite products.

\begin{corollary}\label{cor:finprods}
Let $k$ be a non-negative integer. Then there is a canonical isomorphism of profinite groups $$\widehat{\mathbb{Z}^k} \cong \widehat{\mathbb{Z}}^k.$$
\end{corollary}

We now note a recent result which compares the $\i$-categorical model for profinite spaces just presented with the model categorical approach developed by Quick in \cite{Qu,Qu2}:

\begin{theorem}[{\cite[Corollary 7.4.6]{prohomotopy}}]\label{thm:prohom}
The $\i$-category associated to the model category presented in \cite{Qu,Qu2} is equivalent to $\Profs.$
\end{theorem}

The details of Quick's model category need not concern us here, but we cite the above theorem in order to freely use results of \cite{Qu,Qu2} about profinite spaces.

\begin{proposition}\label{prop:profiniteemergency}
Let $k$ be a non-negative integer. There is a canonical equivalence of profinite spaces
$$\widehat{B\left(\mathbb{Z}^k\right)} \simeq B\left(\widehat{\mathbb{Z}}^k\right).$$
\end{proposition}

\begin{proof}
Since $\mathbb{Z}^k$ is a finitely generated free abelian group, it is \emph{good} in the sense of Serre in \cite{Serre}. It follows from \cite[Proposition 3.6]{Qu2} and Theorem \ref{thm:prohom} that the canonical map $$\widehat{B\left(\mathbb{Z}^k\right)} \to B\left(\widehat{\mathbb{Z}^k}\right)$$ is an equivalence of profinite spaces. The result now follows from Corollary \ref{cor:finprods}.
\end{proof}

The following lemma will be used in an essential way several times in this paper:

\begin{lemma}\label{lem:truncatedcosimplicial}
Let $f:\Delta \to \sC$ be a cosimplicial diagram and suppose that $\sC$ is an $\left(n+1,1\right)$-category, i.e. an $\i$-category whose mapping spaces are all $n$-truncated. Then, provided both limits exist, the canonical map
$$\lim f \to \lim \left(f|_{\Delta{\le n}}\right)$$ is an equivalence.
\end{lemma}

\begin{proof}

Let $\sC$ be an arbitrary $\left(\infty,n+1\right)$-category. Notice that for any diagram $f:\Delta \to \sC,$ and any object $C$ of $\sC,$ we have
$$\Hom\left(C,\underset{k \in \Delta} \lim f(k)\right) \simeq \underset{k \in \Delta} \lim \Hom\left(C, f(k)\right),$$
and since $\sC$ is an $\left(\infty,n+1\right)$-category, each $\Hom\left(C, f(k)\right)$ is an $n$-truncated space. Therefore the general case follows from the case when $\sC$ is the full subcategory $\cS^{\le n}$ of $\cS$ on the $n$-truncated spaces.
By \cite[Theorem 4.2.4.1]{htt}, to prove the lemma for the special case $\sC=\cS^{\le n},$ it suffices to prove the corresponding statement about homotopy limits in the Quillen model structure on the category of compactly generated spaces $\mathbf{CG}$, since the associated $\i$-category is $\cS.$

Suppose that $$X^\bullet:\Delta \to \mathbf{CG}$$ is a cosimplicial space which is fibrant with respect to the projective model structure on $\Fun\left(\Delta,\mathbf{CG}\right)$ (with respect to the Quillen model structure on $\mathbf{CG}$), i.e., the diagram $X^\bullet$ consists entirely of Serre fibrations. Then the homotopy limit of $X^\bullet$ may be computed as $\operatorname{Tot}\left(X\right),$ and moreover, $\operatorname{Tot}\left(X\right)$ can be written as the (homotopy) limit of a tower of fibrations
$$\ldots \to \operatorname{Tot}\left(X\right)_k \to \operatorname{Tot}\left(X\right)_{k-1}\to \ldots \to \operatorname{Tot}\left(X\right)_1\to \operatorname{Tot}\left(X\right)_0=X,$$ where each $\operatorname{Tot}\left(X\right)_k$ is a model for the homotopy limit of $X|_{\Delta_{\le k}}.$ Moreover, the (homotopy) fiber of each map $$\operatorname{Tot}\left(X\right)_k \to \operatorname{Tot}\left(X\right)_{k-1}$$ is homotopy equivalent to the $k$-fold loop space $\Omega^k\left(M^k\left(X^\bullet\right)\right),$ where 

$$M^k X^\bullet = \underset{{{ [k+1]} \twoheadrightarrow {[j]}\atop j \leq k}} \lim X^j$$
is the $k$-th matching object of $X^\bullet$ (see e.g. the introduction of \cite{tot}).

Now, let us assume that each $X_k$ is $n$-truncated. Then as $X^\bullet$ is fibrant, the diagram involved in the limit above consists entirely of fibrations, so the limit is a homotopy limit, hence each matching object is also $n$-truncated (since $n$-truncated objects are stable under limits in $\cS$ by \cite[Proposition 5.5.6.5]{htt}). It follows then that each homotopy fiber $$\operatorname{Tot}\left(X\right)_k \to \operatorname{Tot}\left(X\right)_{k-1}$$ is weakly contractible for $k > n,$ and hence the natural map $$\holim X^\bullet= \operatorname{Tot}\left(X\right) \to \operatorname{Tot}\left(X\right)_n=\holim X^\bullet|_{\Delta_{\le n}}$$ is a weak homotopy equivalence.
\end{proof}

\begin{proposition}\label{prop:infinity-stone}
Let $\underset{i \in \cI} \lim G_i$ be a pro-object in the category of finite groups, or equivalently, a group object in $\Pro\left(\mathit{FinSet}\right).$ Consider the profinite space $$B\left(\underset{i} \lim G_i\right)\defeq  \colim\left[\ldots \underset{i} \lim G_i \times \underset{i} \lim G_i \rrrarrow \underset{i} \lim G_i  \rightrightarrows *\right],$$ where the colimit is computed in $\Profs,$ and $*$ denotes the terminal profinite space. In more detail, the diagram whose colimit is being taken is the simplicial diagram which is the \v{C}ech nerve of the unique map $\underset{i} \lim G_i \to *$ in $\Profs.$ Consider for each $i$ the object in $\cS$ $$B\left(G_i\right)\defeq \colim\left[\ldots G_i \times G_i \rrrarrow G_i  \rightrightarrows *\right],$$ i.e. the colimit in $\cS$ of the \v{C}ech nerve of $G_i.$ Then these spaces assemble into a profinite space $\underset{i} \lim B\left(G_i\right),$ and we have a canonical equivalence
$$B\left(\underset{i} \lim G_i\right) \simeq \underset{i} \lim B\left(G_i\right)$$ in $\Profs.$
\end{proposition}


\begin{proof}
It suffices to prove that for each $\pi$-finite space $V,$ we have an equivalence $$\Hom_{\Profs}\left(B\left(\underset{i} \lim G_i\right),j\left(V\right)\right) \simeq \Hom_{\Profs}\left(\underset{i} \lim B\left(G_i\right),j\left(V\right)\right).$$ Recall that by Proposition \ref{prop:pitruncated} $j\left(V\right)$ is $n$-truncated for some $n.$ As such, we have

\begin{eqnarray*}
\Hom_{\Profs}\left(B\left(\underset{i} \lim G_i\right),j\left(V\right)\right) &\simeq& \Hom_{\Profs}\left(\underset{\Delta^{op}}\colim N\left(\underset{i} \lim G_i\right),j\left(V\right)\right)\\
&\simeq& \underset{\Delta}\lim  \Hom_{\Profs}\left(N\left(\underset{i} \lim G_i\right),j\left(V\right)\right)\\
&\simeq& \underset{\Delta_{\le n}} \lim \Hom_{\Profs}\left(N\left(\underset{i} \lim G_i\right),j\left(V\right)\right),
\end{eqnarray*}
that last equivalence following from Lemma \ref{lem:truncatedcosimplicial}. Expanding this out we get
$$\resizebox{6in}{!}{$\underset{\Delta_{\le n}} \lim \left[\Hom_{\Profs}\left(1,j\left(V\right)\right) \rightrightarrows \Hom_{\Profs}\left(\underset{i} \lim G_i,j\left(V\right)\right) \rrrarrow \Hom_{\Profs}\left(\left(\underset{i} \lim G_i\right)^2,j\left(V\right)\right) \ldots \Hom_{\Profs}\left(\left(\underset{i} \lim G_i\right)^n,j\left(V\right)\right)\right]$}$$
which is equivalent to
$$\resizebox{6in}{!}{$\underset{\Delta_{\le n}} \lim \left[\Hom_{\cS}\left(*,V\right) \rightrightarrows \underset{i} \colim  \Hom_{\cS}\left(G_i,j\left(V\right)\right) \rrrarrow \underset{i} \colim \Hom_{\cS}\left(G_i^2,j\left(V\right)\right) \ldots \underset{i} \colim  \Hom_{\cS}\left(G_i^n,j\left(V\right)\right)\right]$}$$
and since by \cite[Proposition 5.3.3.3]{htt} finite limits commute with filtered colimits in $\cS,$ we get 
$$\resizebox{6in}{!}{$  \underset{i} \colim  \underset{\Delta_{\le n}} \lim \left[\Hom_{\cS}\left(*,V\right) \rightrightarrows\Hom_{\cS}\left(G_i,V\right) \rrrarrow \Hom_{\cS}\left(G_i^2,j\left(V\right)\right) \ldots\Hom_{\cS}\left(G_i^n,V\right)\right].$}$$
Now since $j\left(V\right)$ is $n$-truncated by Proposition \ref{prop:pitruncated}, it follows from Lemma \ref{lem:truncatedcosimplicial} that we can rewrite this as
$$\resizebox{6in}{!}{$  \underset{i} \colim  \underset{\Delta} \lim \left[\Hom_{\cS}\left(*,V\right) \rightrightarrows\Hom_{\cS}\left(G_i,V\right) \rrrarrow \Hom_{\cS}\left(G_i^2,j\left(V\right)\right) \ldots\Hom_{\cS}\left(G_i^n,V\right) \ldots\right]$}$$
which is equivalent to
\begin{eqnarray*}
\underset{i} \colim \Hom_{\cS}\left( \underset{\Delta^{op}} \colim N\left(G_i\right),V\right) &\simeq& \underset{i} \colim \Hom_{\cS}\left( B\left(G_i\right),V\right)\\
&\simeq& \Hom_{\Profs}\left(\underset{i} \lim B\left(G_i\right),j\left(V\right)\right). {\qedhere}
\end{eqnarray*}
\end{proof}


\section{The homotopy type of topological stacks}\label{sec:homotopy}
In this section we use the formalism of $\i$-categories to produce two important constructions necessary for our paper. Firstly, we extend the construction of analytification, which sends a complex variety to its underlying topological space with the complex analytic topology, to a colimit-preserving functor 
$$\left(\blank\right)_{top}:\Shi\left(\Aff,\mbox{\'et}\right) \to \Hshi\left(\TopC\right)$$ from the $\i$-category of $\i$-sheaves over the \'etale site of affine schemes of finite type over $\mathbb{C}$ to the $\i$-category of hypersheaves over an appropriate topological site. This functor, in particular, sends an Artin stack locally of finite type over $\mathbb{C}$ to its underlying topological stack in the sense of Noohi \cite{No1}. Using this functor, one associates to the infinite root stack $\sqrt[\i]{X}$ of a log-scheme a pro-topological stack $\sqrt[\i]{X}_{top}.$ In Section \ref{section: KN to IRS}, we produce a map 
\begin{equation}\label{eq:Knmap}
X_{log} \to \sqrt[\i]{X}_{top}
\end{equation}
from the Kato-Nakayama space to the underlying (pro-)topological stack of the infinite root stack. The main result of the paper is that this map is a profinite homotopy equivalence, but to make sense of such a statement, one first needs to associate to each of these objects a (pro-)homotopy type, in a functorial way. To achieve this, the second construction we produce is a colimit-preserving functor $$\Pi_\i:\Hshi\left(\TopC\right) \to \cS$$ which sends every topological space $X$ to its underlying homotopy type, and sends every topological stack to its homotopy type in the sense of Noohi \cite{No2}. Using this construction and the map (\ref{eq:Knmap}), one has an induced map in $\Pro\left(\cS\right)$ $$\Pi_\i\left(X_{log}\right) \to \Pi_\i\left(\sqrt[\i]{X}_{top}\right)$$ which we prove in Section \ref{section: the equivalence} to become an equivalence after applying the profinite completion functor, i.e. the map (\ref{eq:Knmap}) is a profinite homotopy equivalence.

\subsection{The underlying topological stack of an algebraic stack}
Let $\Top$ be the category of topological spaces and let $\TopCs$ denote the full subcategory of $\Top$ of all contractible and locally contractible spaces which are homeomorphic to a subspace of $\mathbb{R}^n$ for some $n.$  Note that $\TopCs$ is essentially small. Denote by $\TopC$ the following subcategory of topological spaces:

\begin{itemize}
\item[$\bullet$] A topological space $T$ is in $\TopC$ if $T$ has an open cover $\left(U_\alpha \hookrightarrow T\right)_\alpha$ such that each $U_\alpha$ is an object of $\TopCs.$
\end{itemize}


We use the subscript $\mathbb{C}$ to highlight the fact that $\TopC$ 
will serve as the target of the analytification functor from the category of algebraic spaces over $\mathbb{C}$. Note that the objects of $\TopC$ are closed under taking open subspaces. As such, it makes sense to equip $\TopC$ with the Grothendieck topology generated by open covers. Denote by $\Hshi\left(\TopC\right)$ the $\i$-topos of hypersheaves on $\TopC,$ i.e. the hypercompletion of the $\i$-topos of $\i$-sheaves. There is also a natural structure of a Grothendieck site on $\TopCs$ as follows:
\begin{itemize}
\item[$\bullet$] Let $T$ be a space in $\TopCs.$ A covering family of $T$ consists of an open cover $\left(U_\alpha \hookrightarrow T\right)$ such that each $U_\alpha$ is in $\TopCs$.
\end{itemize}
Note that every open cover of $T$ can be refined by such a cover. We denote the associated $\i$-topos of hypersheaves by $\Hshi\left(\TopCs\right).$ By the Comparison Lemma of \cite{SGA4} III, we have that restriction along the canonical inclusion $$\TopCs \hookrightarrow \TopC$$ induces an equivalence between their respective categories of sheaves of sets. It then follows from \cite[Theorem 5]{Jardine} and \cite[Proposition 6.5.2.14 ]{htt} that this lifts to an equivalence $$\Hshi\left(\TopC\right) \stackrel{\sim}{\longrightarrow} \Hshi\left(\TopCs\right),$$ and in particular, $\Hshi\left(\TopC\right)$ is an $\i$-topos (which is not a priori clear for sites which are not essentially small).

Denote by $\Aff$ the category of affine schemes of finite type over $\mathbb{C}.$ Note that it is a small category with finite limits. Denote by $$\left(\blank\right)_{an}:\Aff \to \Top$$ the functor associating to such an affine scheme its space of $\mathbb{C}$-points, equipped with the analytic topology. The above functor preserves finite limits, and is the restriction of a functor defined for all algebraic spaces locally of finite type over $\mathbb{C},$ see \cite[p. 12]{toen-vaquie}. Note also that if $V$ is a scheme which is separated and locally of finite type, 
then $V_{an}$ is locally (over any affine) a triangulated space by \cite{Lo}, so in particular $V_{an}$ is locally contractible. Also observe that $V_{an}$ is locally cut-out of $\mathbb{C}^n$ by polynomials, so it follows that $V_{an}$ is in $\TopC.$ 
Consequently $\left(\blank\right)_{an}$ restricts to a functor $$\left(\blank\right)_{an}:\Aff \to \TopC,$$ which preserves finite limits.

Note that the category $\Aff$ can be equipped with the Grothendieck topology generated by \'etale covering families. Denote the associated $\i$-topos of $\i$-sheaves on this site by $\Shi\left(\Aff,\mbox{\'et}\right).$


%

The following theorem is an extension of \cite[Proposition 20.2]{No1}:

\begin{theorem}\label{thm:analification}
The functor $$\left(\blank\right)_{an}:\Aff \to \TopC$$ lifts to a left exact colimit preserving functor $$\left(\blank\right)_{top}:\Shi\left(\Aff,\mbox{\'et}\right) \to \Hshi\left(\TopC\right).$$
\end{theorem}

\begin{proof}
Note that the image under $\left(\blank\right)_{an}$ of an \'etale map is a local homeomorphism. Also note that if $$S \to T$$ is a local homeomorphism and $T$ is in $\TopC,$ so is $S.$ Furthermore, since the inclusion of any open subspace of a topological space is a local homeomorphism, and since any cover by local homeomorphisms can be refined by a cover by open subspaces, it follows that open covers and local homeomorphisms generate the same Grothendieck topology on $\TopC.$ It follows that any $\i$-sheaf on $\TopC$, so in particular any hypersheaf, satisfies descent with respect to covers by local homeomorphisms. The result now follows from \cite[Proposition 6.2.3.20]{htt}.
\end{proof}

\begin{remark}
Denote by $Y$ the Yoneda embedding $$Y:\TopC \hookrightarrow \Hshi\left(\TopC\right)$$ and denote by $y$ the Yoneda embedding $$y:\Aff \hookrightarrow \Shi\left(\Aff,\mbox{\'et}\right).$$ Explicitly, $\left(\blank\right)_{top}$ is the left Kan extension of $Y \circ \left(\blank\right)_{an}$ along $y,$ $$\Lan_y\left[Y\circ \left(\blank\right)_{an}\right]:\Shi\left(\Aff,\mbox{\'et}\right) \to \Hshi\left(\TopC\right),$$ or more concretely, it is the unique colimit preserving functor such that for a representable $y\left(X\right),$ i.e. an affine scheme, $$y\left(X\right)_{top}\cong Y\left(X_{an}\right).$$
\end{remark}

By the proof of Theorem \ref{thm:analification}, we see that given any hypersheaf $F$ on $\TopC,$ the functor $$F \circ \left(\blank\right)_{an}$$ is an $\i$-sheaf on $\left(\Aff,\mbox{\'et}\right),$ i.e. we have a well-defined functor
$$\left(\blank\right)_{an}^{*}:\Hshi\left(\TopC\right) \to \Shi\left(\Aff,\mbox{\'et}\right).$$

\begin{proposition}
The functor $\left(\blank\right)_{top}$ is left adjoint to $\left(\blank\right)_{an}^{*}$.
\end{proposition}

\begin{proof}
Since $\left(\blank\right)_{top}$ is colimit preserving, it follows from \cite[Corollary 5.5.2.9]{htt} that it has a right adjoint. Let us denote the right adjoint by $R.$ By the Yoneda lemma, we have that if $F$ is a hypersheaf $F$ on $\TopC,$ then $R\left(F\right)$ is the $\i$-sheaf on $\left(\Aff,\mbox{\'et}\right)$ such that if $X$ is an affine scheme,
\begin{eqnarray*}
R\left(F\right)\left(X\right) &\simeq& \Hom\left(y\left(X\right),R\left(F\right)\right)\\
&\simeq& \Hom\left(\left(y\left(X\right)\right)_{top},F\right)\\
&\simeq& \Hom\left(Y\left(X_{an}\right),F\right)\\
&\simeq& F\left(X_{an}\right).
\end{eqnarray*}
\end{proof}

\begin{remark}\label{rmk:geometric morphism}
The adjoint pair $\left(\blank\right)_{top} \dashv \left(\blank\right)_{an}^{*}$ assemble into a geometric morphism of $\i$-topoi $$f:\Hshi\left(\TopC\right) \to \Shi\left(\Aff,\mbox{\'et}\right),$$ with direct image functor $$f_*=\left(\blank\right)_{an}^{*}$$ and inverse image functor $$f^*=\left(\blank\right)_{top}.$$
\end{remark}

\begin{lemma}
Let $\As$ denote the category of algebraic spaces locally of finite type over $\mathbb{C}.$ Equip $\As$ with the \'etale topology. Then restriction along the canonical inclusion $$\Aff \hookrightarrow \As$$ induces an equivalence of $\i$-categories
$$\Shi\left(\As,\mbox{\'et}\right) \stackrel{\sim}{\longrightarrow} \Shi\left(\Aff,\mbox{\'et}\right).$$
\end{lemma}

\begin{proof}
The inclusion satisfies the conditions of the Comparison Lemma of \cite{SGA4} III, so we have an induced equivalence $$\Sh\left(\As,\mbox{\'et}\right) \stackrel{\sim}{\longrightarrow} \Sh\left(\Aff,\mbox{\'et}\right)$$ between sheaves of sets. Since both sites have finite limits, the result now follows from \cite[Proposition 6.4.5.4]{htt}.
\end{proof}

\begin{proposition}\label{prop:alspacan}
Let $X$ be any algebraic space locally of finite type over $\mathbb{C}.$ Then $X_{top} \cong X_{an}.$
\end{proposition}

\begin{proof}
Let $\cU$ be the Grothendieck universe of small sets and let $\cV$ be the Grothendieck universe of large sets with $\cU \in \cV.$ Denote by $\LiGpd$ the $\i$-category of large $\i$-groupoids, and denote by $\widehat{\Hshi}\left(\TopC\right)$ the $\i$-category of hypersheaves on $\TopC$ with values in $\LiGpd,$ and similarly let $\widehat{\Shi}\left(\As,\mbox{\'et}\right)$ denote the $\i$-category of sheaves on the \'etale site of algebraic spaces with values in $\LiGpd$. Then by same proof as Theorem \ref{thm:analification}, by left Kan extension there is a $\mathcal{V}$-small colimit preserving functor $$L:\widehat{\Shi}\left(\As,\mbox{\'et}\right) \to \widehat{\Hshi}\left(\TopC\right)$$ such that for all representable  sheaves $y\left(P\right)$ on $\left(\As,\mbox{\'et}\right),$ $$L\left(y\left(P\right)\right)\cong Y\left(P_{an}\right).$$ By \cite[Remark 6.3.5.17]{htt}, both inclusions $$\Hshi\left(\TopC\right) \hookrightarrow \widehat{\Hshi}\left(\TopC\right)$$ and $$\Shi\left(\Aff,\mbox{\'et}\right)\hookrightarrow \widehat{\Shi}\left(\Aff,\mbox{\'et}\right)$$ preserve $\cU$-small colimits. Hence both composites
$$\Shi\left(\Aff,\mbox{\'et}\right)\hookrightarrow \widehat{\Shi}\left(\Aff,\mbox{\'et}\right)\simeq  \widehat{\Shi}\left(\As,\mbox{\'et}\right) \stackrel{L}{\longrightarrow} \widehat{\Hshi}\left(\TopC\right)$$ and $$\Shi\left(\Aff,\mbox{\'et}\right) \stackrel{\left(\blank\right)_{top}}{\longlongrightarrow} \Hshi\left(\TopC\right) \hookrightarrow \widehat{\Hshi}\left(\TopC\right)$$ are  $\mathcal{cU}$-small colimit preserving, and agree up to equivalence on every representable $y\left(X\right)$, for $X$ an affine scheme. It follows from \cite[Theorem 5.1.5.6]{htt} that both compositions must in fact be equivalent. However, the inclusion $$\Shi\left(\Aff,\mbox{\'et}\right)\hookrightarrow \widehat{\Shi}\left(\Aff,\mbox{\'et}\right)\simeq  \widehat{\Shi}\left(\As,\mbox{\'et}\right)$$ carries an algebraic space $P$ to its representable sheaf $y\left(P\right)$. The result now follows.
\end{proof}

The following lemma follows immediately from the fact that $\left(\blank\right)_{top}$ preserves finite limits:

\begin{lemma}
Let $\cG$ be a groupoid object in sheaves of sets on the \'etale site $\left(\Aff,\mbox{\'et}\right)$. Then applying $\left(\blank\right)_{top}$ level-wise produces a groupoid object in sheaves of sets on $\TopC$, denoted by $\cG_{top}$. Moreover, if the original groupoid $\cG$ is a groupoid object in algebraic spaces, then $\cG_{top}$ is degree-wise representable, i.e. a topological groupoid.
\end{lemma}

\begin{proposition}\label{prop:gpdobjs}
Let $\cG$ be a groupoid object in sheaves of sets on the \'etale site $\left(\Aff,\mbox{\'et}\right).$ Denote by $\left[\cG\right]$ its associated stack of torsors, and denote by $\left[\cG_{top}\right]$ the stack of groupoids on $\TopC$ associated to $\cG_{top},$ i.e. stack on $\TopC$ of principal $\cG_{top}$-bundles. Then $\left[\cG\right]_{top} \simeq \left[\cG_{top}\right].$
\end{proposition}
\begin{proof}
The stack $\left[\cG\right]$ is the stackification of the presheaf of groupoids $\tilde y\left(\cG\right)$ which sends an affine scheme $X$ to the groupoid $\cG\left(X\right).$ Denote by $N\left(\cG\right)$ the simplicial presheaf which is the nerve of this presheaf of groupoids. Consider the diagram $$\Delta^{op} \stackrel{N\left(\cG\right)}{\longlongrightarrow} \Psh\left(\Aff,\mathit{Set}\right) \hookrightarrow \Psh\left(\Aff,\iGpd\right).$$ We claim that the colimit of the above functor is $\tilde y\left(\cG\right).$ Since colimits are computed object-wise, it suffices to show that if $\cH$ is any discrete groupoid, then $N\left(\cH\right)$ is the homotopy colimit of the diagram $$\Delta^{op} \stackrel{N\left(\cH\right)}{\longlongrightarrow} \mathit{Set} \hookrightarrow \mathit{Set}^{{\Delta}^{op}},$$ which follows easily from the well-known fact that the homotopy colimit of a simplicial diagram of simplicial sets can be computed by taking the diagonal. It follows then that $\left[\cG\right]$ is the colimit of the diagram $$\Delta^{op} \stackrel{N\left(\cG\right)}{\longlongrightarrow} \Sh\left(\Aff,\mbox{\'et}\right) \hookrightarrow \Shi\left(\Aff,\mbox{\'et}\right),$$ since $\i$-sheafification preserves colimits, as it is a left adjoint. By the same argument, we have that $\left[\cG_{top}\right]$ is the colimit of the diagram 
$$\Delta^{op} \stackrel{N\left(\cG_{top}\right)}{\longlongrightarrow} \Sh\left(\TopC\right) \hookrightarrow \Shi\left(\TopC\right).$$ Notice that for all $n$ we have $$N\left(\cG_{top}\right)_n=\left(N\left(\cG\right)_n\right)_{top}.$$ The result now follows from the fact that $\left(\blank\right)_{top}$ preserves colimits.
\end{proof}

\begin{definition}
A \textbf{topological stack} is a stack on $\TopC$ of the form $\left[\cG\right]$ for $\cG$ a groupoid object in $\TopC.$ Denote the associated $\left(2,1\right)$-category of topological stacks by $\mathfrak{TopSt}.$
\end{definition}

\begin{remark}\label{rmk: loc stack}
In the literature, typically there is no restriction on a topological stack to come from a topological groupoid which is locally contractible, and such a stack is represented by its functor of points on the Grothendieck site of all topological spaces. However, the $\left(2,1\right)$-category of topological stacks in the sense we defined above embeds fully faithfully into the larger $\left(2,1\right)$-category of all topological stacks in the classical sense.
\end{remark}

\begin{corollary}\label{cor:algstacktotopstack}
The functor $$\left(\blank\right)_{top}:\Shi\left(\Aff,\mbox{\'et}\right) \to \Hshi\left(\TopC\right)$$ restricts to a left exact functor $$\left(\blank\right)_{top}:\mathbf{A}\!\mathfrak{lgSt}^{LFT}_{\mathbb{C}} \to \mathfrak{TopSt}$$ from Artin stacks locally of finite type over $\mathbb{C}$ to topological stacks.
\end{corollary}

Up to the identification mentioned in Remark \ref{rmk: loc stack}, the construction in the above corollary agrees with that of Noohi in \cite[Section 20]{No1}.

\subsection{The fundamental infinity-groupoid of a stack} \mbox{}


The following proposition will allow us to talk about homotopy types of topological stacks.

\begin{proposition}\label{prop:minipi}
There is a colimit preserving functor $$\overline{L}:\Hshi\left(\TopC^s\right) \to \cS$$ sending every representable sheaf $y\left(T\right)$ for $T$ a space in $\TopC^s$ to its weak homotopy type.
\end{proposition}

\begin{proof}
The proof is essentially the same as \cite[Proposition 3.1]{hafdave}. By Lemma 3.1 in op. cit., there is a functor $$\TopC^s \hookrightarrow \Top \stackrel{h}{\longrightarrow} \cS$$ assigning to each space $T$ its associated weak homotopy type. Denote this functor by $\pi$. Since $\TopC^s$ is essentially small, by left Kan extension there is a colimit preserving functor $$\Lan_Y \pi:\Pshi\left(\TopC^s\right) \to \cS$$ sending every representable presheaf $Y\left(T\right)$ to the underlying weak homotopy type of $T$. It follows from the Yoneda lemma that this functor has a right adjoint $R_\pi$ which sends an $\i$-groupoid $Z$ to the $\i$-presheaf $$R_\pi\left(Z\right):T \mapsto \Hom\left(\pi\left(T\right),Z\right).$$ We claim that $R_\pi\left(Z\right)$ is a hypersheaf. To see this, it suffices to observe that if $$U^{\bullet}:\Delta^{op} \to \TopC^s/T$$ is a hypercover of $T$ with respect to the coverage of contractible open coverings, then the colimit of the composite $$\Delta^{op} \stackrel{U^\bullet}{\longlongrightarrow} \TopC^s/T \to \TopC^s \stackrel{\pi}{\longrightarrow} \cS$$ is $\pi\left(T\right),$ which follows from \cite[Theorem 1.3]{duggerisaksen}. We thus have that $\Lan_y \pi$ and $R_\pi$ restrict to an adjunction $$\overline{L} \dashv \Delta$$ between $\Hshi\left(\TopC^s\right)$ and $\cS,$ so in particular, $\overline{L}$ preserves colimits.
\end{proof}

\begin{corollary}\label{cor:constsheaf}
Let $\cG$ be an $\i$-groupoid. Denote by $\Delta\left(\cG\right)$ the constant presheaf on $\TopC^s.$ Then $\Delta\left(\cG\right)$ is a hypersheaf.
\end{corollary}

\begin{proof}
Following the proof of the above theorem, we have that $R_\pi\left(\cG\right)$ is a hypersheaf. Moreover, for each space $T$ in $\TopC^s,$ we have that
\begin{eqnarray*}
R_\pi\left(\cG\right)\left(T\right) &\simeq& \Hom\left(Y\left(T\right),R_\pi\left(\cG\right)\right)\\
&\simeq& \Hom\left(\overline{L}\left(Y\left(T\right)\right),\cG\right)\\
&\simeq& \Hom\left(*,\cG\right)\\
&\simeq& \cG,
\end{eqnarray*}
since each such $T$ is in fact contractible.
\end{proof}

\begin{remark}\label{rmk:delta}
The $\i$-category $\cS$ of spaces is the terminal $\i$-topos. In particular, if $\sC$ is any $\i$-category equipped with a Grothendieck topology, then the unique geometric morphism $$\Shi\left(\sC\right) \to \cS$$ has as direct image functor the global sections functor $$\Gamma:\Shi\left(\sC\right) \to \cS$$ defined by $\Gamma\left(F\right)=\Hom\left(1,F\right),$ which is the same as $F\left(1\right)$ if $\sC$ has a terminal object. The inverse image functor is given by $$\Delta:\cS \to \Shi\left(\sC\right)$$ and it sends an $\i$-groupoid $\cG$ to the sheafification of the constant presheaf with value $\cG$. Similarly, the unique geometric morphism $$\Hshi\left(\sC\right) \to \cS$$ has its direct image functor  $\Gamma$ given by the same construction as for $\i$-sheaves, and the inverse image functor $\Delta$ assigns an $\i$-groupoid $\cG$ the hypersheafification of the constant presheaf with values $\cG$. In either case we have $\Delta \dashv \Gamma.$ In particular, Corollary \ref{cor:constsheaf} implies that for the $\i$-topos $\Hshi\left(\TopC^s\right)$ we have a triple of adjunctions: $$\overline{L} \dashv \Delta \dashv \Gamma.$$ Although we will not prove it here, there is in fact a further right adjoint to $\Gamma,$ $coDisc \vdash \Gamma$ and moreover the quadruple $$\overline{L} \dashv \Delta \dashv \Gamma \dashv coDisc$$ exhibits $\Hshi\left(\TopC^s\right)$ as a \emph{cohesive $\i$-topos} in the sense of \cite{schreiber}. 
\end{remark}

\begin{proposition}\label{prop:bigpi}
The composite $$\Hshi\left(\TopC\right) \stackrel{\sim}{\longrightarrow} \Hshi\left(\TopC^s\right) \stackrel{\overline{L}}{\longrightarrow} \cS$$ is colimit preserving and sends a representable sheaf $Y\left(X\right)$, for $X$ in $\TopC$, to its underlying weak homotopy type.
\end{proposition}
\begin{proof} 
By \cite[Lemma 3.1]{hafdave}, there is a functor $$\TopC \hookrightarrow \Top \stackrel{h}{\longrightarrow} \cS$$ assigning to each space $X$ its associated weak homotopy type. Denote this functor by $\Pi$. By exactly the same proof as Proposition \ref{prop:minipi}, by using that $\TopC$ is $\cV$-small, with $\cV$ the Grothendieck universe of large sets, we construct a colimit preserving functor $$\mathbb{L}:\widehat{\Hshi}\left(\TopC\right) \to \widehat{\cS},$$ where $\widehat{\cS}$ is the $\i$-category of large spaces (or large $\i$-groupoids), which sends every representable sheaf $Y\left(X\right)$ to its underlying weak homotopy type. The rest of the proof is analogous to that of Proposition \ref{prop:alspacan}.
\end{proof}

\begin{definition}
We denote the composite from Proposition \ref{prop:bigpi} by $$\Pi_\i:\Hshi\left(\TopC\right) \to \cS.$$ For $F$ a hypersheaf on $\TopC,$ we call $\Pi_\i\left(F\right)$ its \textbf{fundamental $\i$-groupoid}. 
\end{definition}

\begin{remark}
In light of Remark \ref{rmk:delta}, we have that $\Pi_\i \dashv \Delta \dashv \Gamma,$ where $\Gamma$ is global sections, and $\Delta$ assigns an $\i$-groupoid $\cG$ the hypersheafification of the constant presheaf with value $\cG$. In particular, we have a formula for $\Delta\left(\cG\right),$ namely, if $X$ is a space in $\TopC,$ $$\Delta\left(\cG\right)\left(X\right)\simeq \Hom\left(\Pi_\i\left(X\right),\cG\right),$$ that is, the space of maps from the homotopy type of $X$ to $\cG.$
\end{remark}

The following proposition may be seen as an extension of the results of \cite{No2}:

\begin{proposition}\label{prop:fat}
Let $\cG$ be a groupoid object in $\TopC$ and denote by $\left[\cG\right]$ denote the associated stack of groupoids on $\TopC$, i.e. the stack of principal $\cG$-bundles. Then $\Pi_\i\left(\left[\cG\right]\right)$ has the same weak homotopy type as $$B\cG=||N\left(\cG\right)||$$- the fat geometric realization of the simplicial space arising as the topologically enriched nerve of $\cG.$
\end{proposition}


\begin{proof}
We know that $\left[\cG\right]$ is the colimit in $\Hshi\left(\TopC\right)$ of the diagram $$\Delta^{op} \stackrel{N\left(\cG\right)}{\longlongrightarrow} \TopC \stackrel{Y}{\hookrightarrow} \Hshi\left(\TopC\right)$$ (as in the proof of Proposition \ref{prop:gpdobjs}). The result now follows from Proposition \ref{prop:bigpi} and \cite[Lemma 3.3]{hafdave}
\end{proof}

\begin{lemma}
Let $F$ be a hypersheaf on $\TopC^s.$ Then $\overline{L}\left(F\right)$ is the colimit of $F,$ i.e. the colimit of the diagram $$F:\left(\TopC^s\right)^{op} \to \cS.$$
\end{lemma}

\begin{proof}
By the proof of Proposition \ref{prop:minipi}, $\overline{L}$ factors as the composition $$\Hshi\left(\TopC^s\right) \hookrightarrow \Pshi\left(\TopC^s\right) \stackrel{\Lan_Y \pi}{\longlonglongrightarrow} \cS=\iGpd.$$ Note however that every space in $\TopC^s$ is contractible, so the canonical morphism $\pi \to t,$ to the terminal functor $$t:\TopC^s \to \iGpd$$ (i.e. the functor with constant value the one point set), is an equivalence, and hence $\Lan_Y \pi$ is left adjoint to the constant functor $t^*$ which sends an $\i$-groupoid $\cG$ to the constant presheaf with value $\cG$. Since $\Pshi\left(\TopC^s\right)=\Fun\left(\left(\TopC^s\right)^{op},\iGpd\right),$ the result now follows from the universal property of $\colim\left(\blank\right).$
\end{proof}

\begin{corollary}\label{cor:colimhomotopy}
Let $F$ be a hypersheaf on $\TopC$. Then $\Pi_\i\left(F\right)$ is the colimit of $F|_{\TopC^s}.$
\end{corollary}

\subsection{The profinite homotopy type of a (pro-)stack}\label{sec:profhomtypestacks}

Let us define the profinite version of the homotopy type of a stack.

\begin{definition}
We denote the composite $$\Hshi\left(\TopC\right) \stackrel{\Pi_\i}{\longlongrightarrow} \cS \stackrel{\widehat{\left(\blank\right)}}{\longlongrightarrow} \Profs$$ by $\widehat{\Pi}_\i.$ For $F$ a hypersheaf on $\TopC,$ we call $\widehat{\Pi}_\i\left(F\right)$ its \textbf{profinite fundamental $\i$-groupoid} or simply its \textbf{profinite homotopy type}. 
\end{definition}

Let us extend the constructions of this section to pro-objects. Note that the functor $$\left(\blank\right)_{top}:\Shi\left(\Aff,\mbox{\'et}\right) \to \Hshi\left(\TopC\right)$$ extends to a functor on pro-objects, which by abuse of notation, we will denote by the same symbol
$$\left(\blank\right)_{top}:\Pro\left(\Shi\left(\Aff,\mbox{\'et}\right)\right) \to \Pro\left(\Hshi\left(\TopC\right)\right).$$
We now describe how to define the profinite homotopy type of a pro-object in $\Hshi\left(\TopC\right)$. First, we may extend the profinite fundamental $\i$-groupoid functor on $\Hshi\left(\TopC\right)$ to pro-objects. This can be achieved easily by the universal property of $\Pro\left(\Hshi\left(\TopC\right)\right).$ Indeed, consider the functor $$\widehat{\Pi}_\i:\Hshi\left(\TopC\right) \to \Profs$$ and denote its unique cofiltered limit preserving extension, by abuse of notation, again by $$\widehat{\Pi}_\i:\Pro\left(\Hshi\left(\TopC\right)\right) \to \Profs.$$ Unwinding the definitions, we see that if $\underset{i \in \cI} \lim \cY^i$ is a pro-object in hypersheaves on $\TopC,$ then its profinite homotopy type is $$\widehat{\Pi}_\i\left(\underset{i \in \cI} \lim \cY^i\right)=\underset{i \in \cI} \lim \widehat{\Pi}_\i\left(\cY^i\right).$$ 



\subsection{The homotopy type of Kato-Nakayama spaces}\label{subsec.KN}

In this subsection, we will give a formula expressing the homotopy type of the Kato-Nakayama space of a log scheme in terms of algebro-geometric data. We first start by reviewing a functorial approach to Kato-Nakayama spaces which is due to Kato, Illusie and Nakayama. Let $ \left( X, M, \alpha \right)$ be a log scheme, and  let $X_{an}$ be the analytification of $ X$, which is an object of $\TopC$.



Consider the slice category $\TopC/ X_{an}.$ If $\left(T \stackrel{p}{\rightarrow}  X_\an\right)$ is an object in $\TopC/ X_{an}$, one can pullback $M$ to 
$T$ and take the section-wise group completion. In this way we obtain a sheaf of abelian groups on $T$ that we denote $p^*M^{\gp}$. Note that $p^*M^{\gp}$ contains  $p^*\cO_X^\times$ as a sub-sheaf of abelian groups.

Let $G$ be any abelian topological group. If $T$ is a topological space, we denote $\underline{G}_T$ the sheaf on $T$ of continuous maps to $G$ equipped with the group structure coming from addition in $G$. Note that we have 
$f^*\left(\underline{G}_S\right)=\underline{G}_T.$


\begin{definition}


We denote by $F_{log}$ the presheaf of sets on $\TopC/ X_{an}$ that is defined on objects by the following assignment: 
$$ 
\left(T \stackrel{p}{\rightarrow} X_{an}\right) \mapsto \left\{ \text{morphisms of sheaves } s: p^*M^{\gp} \rightarrow \underline{S}^1_T \text{ such that } s(f)=\frac{f}{|f|} \text{ for } f\in \cO_X^\times  \right\}.
$$
\end{definition}

\begin{theorem}[{{\cite[Section 1.2]{illusie-kato-nakayama}}}]
The presheaf $F_{log}$ is represented by $X_{log}$.
\end{theorem}


Since $X_{log}$ is an object of $\TopC,$ the functor $F_{\log}$ completely determines $X_{log}.$ Moreover, we can use this functorial description to give an expression for the homotopy type of $X_{log},$ as we will now show.

\begin{definition}
Denote by $\sC_{KN}\left(X\right)$ the following category:
the objects consist of triples $\left(T,p,s\right)$ where

\begin{itemize}
\item $T$ is a topological space in $\TopC^s,$
\item $p$ is a continuous map $p:T \to X_{\an},$
\item and $s$ is a morphism of sheaves of abelian groups
$$s: p^*M^{\gp} \rightarrow \underline{S}^1_T$$ such that $s(f)=\frac{f}{|f|}$ for  $f\in \cO_X^\times$.
\end{itemize}
The morphisms $\left(T,p,s\right) \to \left(S,q,r\right)$ are continuous maps $f:T \to S$ such that $f^*\left(r\right)=s.$
\end{definition}

\begin{theorem}\label{thm:KN homotopy}
Let $X$ be a log scheme. The weak homotopy type of the Kato-Nakayama space is that of $B\sC_{KN}\left(X\right).$
\end{theorem}

\begin{proof}
The reader may notice that $\sC_{KN}\left(X\right)$ is nothing but the Grothendieck construction

$$\int_{\TopC^s} \left(F_{\log}|_{\TopCs/X_{an}}\right).$$ 
Notice also that $$\TopC^s/X_{an} \to \TopC^s$$ is the Grothendieck construction of $Y\left(X_{an}\right)|_{\TopC^s}$ (where $Y$ denotes the Yoneda embedding) i.e. the corresponding fibered category. Now, there is a canonical equivalence of categories $$\Sh\left(\TopCs/X_{an}\right)\simeq \Sh\left(\TopCs\right)/ Y\left(X_{an}\right)|_{\TopC^s},$$ and it follows that $\int_{\TopC^s} \left(F_{\log}|_{\TopCs/X_{an}}\right)$ is equivalent to the Grothendieck construction of $Y\left(X_{log}\right).$ By Proposition \ref{prop:bigpi}, we have that $\Pi_\i\left(Y\left(X_{log}\right)\right)$ is the weak homotopy type of $X_{log}.$ The result now follows from Corollary \ref{cor:colimhomotopy} and \cite[Corollary 3.2]{hafdave}.
\end{proof}

\begin{corollary}
Let $X$ be a log scheme. The profinite homotopy type of its Kato-Nakayama space $X_{log}$ is that of the profinite completion of $B\sC_{KN}\left(X\right).$
\end{corollary}


\section{Construction of the map}\label{section: KN to IRS}

In all that follows $X$ will be a fine and saturated log scheme over $\bC$ that is locally of finite type. See Appendix \ref{preliminaries} for a condensed introduction to the main concepts and notations that we will use in this section. Our goal is to prove the following proposition:

\begin{proposition}\label{kntoirs}
There is a canonical morphism of pro-topological stacks $$\Phi_X\colon X_\log\to (\sqrt[\infty]{X})_{\topst}.$$
\end{proposition}

Later (Section \ref{section: the equivalence}) we will show that this map induces a weak equivalence of profinite homotopy types. The proof of Proposition \ref{kntoirs} will take up the rest of this section.

Our strategy will be to construct the morphism $\Phi_X$ \'{e}tale locally on $X$, where the log structure has a Kato chart, and then to show that the locally defined morphisms glue together to give a global one.


\textbf{Step 1} (local case): first let us assume that $X\to\Spec \bC[P]$ is a Kato chart for $X$, where $P$ is a fine saturated sharp monoid. In this case everything is very explicit: as explained in Section \ref{subsection.root} in the appendix, there is an isomorphism $\sqrt[n]{X}\simeq [X_n/\mu_n(P)]$, where $X_n=X\times_{\Spec \bC[P]} \Spec \bC[\frac{1}{n}P]$, the group $\mu_n(P)$ is the Cartier dual of the cokernel of $P^\gp\to\frac{1}{n}P^\gp$, and the action on $X_n$ is induced by the natural one on $\Spec \bC[\frac{1}{n}P]$.


By following Noohi's construction (see \ref{noohi}) we see that $\sqrt[n]{X}_\topst$ is canonically isomorphic to the quotient $[(X_n)_\an/\mu_n(P)_\an]$, where $\mu_n(P)_\an\cong (\bZ/n)^r$. Note that the finite morphism $\Spec \bC[\frac{1}{n}P]\to\Spec \bC[P]$ is \'{e}tale on the open torus $\Spec \bC[P^\gp]\subseteq \Spec \bC[P]$, and ramified exactly on the complement. 

Now let us construct a morphism of topological stacks $X_\log\to \sqrt[n]{X}_\topst$. By the quotient stack description of the target, this is equivalent to giving a $\mu_n(P)_\an$-torsor (i.e. principal bundle) on $X_\log$, together with a $\mu_n(P)_\an$-equivariant map to $(X_n)_\an$.

Let us look at a couple of examples first.

\begin{example}
Let $X$ be the standard log point $\Spec \bC$ with log structure given by $\bN\oplus \bC^\times\to \bC$ sending $(n,a)$ to $0^n\cdot a$. Then $X_\log\cong S^1$, and $\sqrt[n]{X}_\topst \simeq \class (\bZ/n)$. In this case the morphism $S^1\to \class (\bZ/n)$ corresponds to the $(\bZ/n)$-torsor $S^1\to S^1$ defined by $z\mapsto z^n$.
\end{example}

\begin{example}\label{example:a1}
Let $X$ be $\bA^1$ with the divisorial log structure at the origin. Then $X_\log \cong \bR_{\geq 0}\times S^1$ and $\sqrt[n]{X}_\topst \simeq [\bC /(\bZ/n)]$, where the morphism  $[\bC /(\bZ/n)]\to (\bA^1)_\an=\bC$ is induced by $z\mapsto z^n$, and $(\bZ/n)$ acts by roots of unity.

In this case the map $\bR_{\geq 0}\times S^1\to [\bC /(\bZ/n)]$ corresponds to the $(\bZ/n)$-torsor $\bR_{\geq 0}\times S^1\to \bR_{\geq 0}\times S^1$ defined by $(a,b)\mapsto (a^n,b^n)$ and the equivariant map $\bR_{\geq 0}\times S^1\to \bC$ given by $(a,b)\mapsto a\cdot b$.

Note that the map $\bR_{\geq 0}\times S^1\to \bR_{\geq 0}\times S^1$ coincides with $z\mapsto z^n$ outside of the ``origin'' $\{0\}\times S^1$, and this is \'{e}tale even on the algebraic side. Over the ``origin'', it is precisely the presence of the $S^1$ introduced by the Kato-Nakayama construction that allows the map to be a $(\bZ/n)$-torsor. This is what happens in general (see also \cite[Lemma 2.2]{KN}).
\end{example}

\begin{proposition}\label{torsor}
Consider the map $\phi_\log \colon (X_n)_\log\to X_\log$ induced by the morphism of log schemes $\phi\colon X_n\to X$. The map $\phi_\log$ is a $\mu_n(P)_\an$-torsor, and the projection $(X_n)_\log \to (X_n)_\an$ is a $\mu_n(P)_\an$-equivariant map.
\end{proposition}



Note (Definition \ref{def.monoid}) that if $P$ is a monoid, $\bC(P)$ will denote the log analytic space $(\Spec \bC[P])_\an$ with the induced natural log structure.

\begin{proof}
The action of $\mu_n(P)$ on $\Spec \bC[\frac{1}{n}P]$ induces an action on $X_n$, and the map $X_n\to X$ is invariant. Consequently we have an induced action of $\mu_n(P)_\an$ on $(X_n)_\log$, and the map $\phi_\log \colon (X_n)_\log\to X_\log$ is invariant.

Moreover, since taking $(\blank)_\log$ commutes with strict base change (see Proposition \ref{kn.basechange}), we have a cartesian diagram
$$
\xymatrix{
(X_n)_\log\ar[r]\ar[d]_{\phi_\log} & \bC(\frac{1}{n}P)_\log\ar[d]^{\phi_{P,\log}}\\
X_\log\ar[r] & \bC(P)_\log
}
$$
and because the action of $\mu_n(P)_\an$ on $(X_n)_\log$ comes from the one on  $\bC(\frac{1}{n}P)_\log$, it suffices to prove the statement for the right-hand column.

Similarly, in order to verify that $(X_n)_\log \to (X_n)_\an$ is $\mu_n(P)_\an$-equivariant we are reduced to checking that  $\bC(\frac{1}{n}P)_\log\to  \bC(\frac{1}{n}P)$ is $\mu_n(P)_\an$-equivariant.

Now note that $\mu_n(P)_\an$ is precisely the kernel of the map $\Hom(\frac{1}{n}P,S^1)\to \Hom(P,S^1)$, So that the action of $\mu_n(P)_\an$ on $\Hom(\frac{1}{n}P,S^1)$ is free and transitive. It is also not hard to check that there are local sections (note that $\Hom(P,S^1)=\Hom(P^\gp,S^1)\cong (S^1)^k$ non-canonically), so that the map is a $\mu_n(P)_\an$-torsor.

Furthermore, $\phi_{P,\log}\colon \bC(\frac{1}{n}P)_\log\to \bC(P)_\log$ is the restriction map $\Hom(\frac{1}{n}P,  \bR_{\geq 0}\times S^1)\to \Hom(P, \bR_{\geq 0}\times S^1)$, and this is the product of the two maps $\Hom(\frac{1}{n}P,  \bR_{\geq 0})\to  \Hom(P,  \bR_{\geq 0})$ (which is a homeomorphism) and $\Hom(\frac{1}{n}P,  S^1)\to  \Hom(P,  S^1)$. The action of $\mu_n(P)_\an$ is trivial on the factor $\Hom(\frac{1}{n}P,  \bR_{\geq 0})$ and the one given by the aforementioned inclusion as a subgroup, on the factor $\Hom(\frac{1}{n}P,  S^1)$. Consequently, $\phi_{P,\log}$ is a $\mu_n(P)_\an$-torsor for the natural action, as required.

The map $\bC(\frac{1}{n}P)_\log\to  \bC(\frac{1}{n}P)$ coincides with the map $\Hom(\frac{1}{n}P, \bR_{\geq 0}\times S^1)\to \Hom(\frac{1}{n}P,\bC)$ induced by the natural map $\bR_{\geq 0}\times S^1\to \bC$, and therefore it is manifestly $\Hom(\frac{1}{n}P,S^1)$-equivariant, and in particular $\mu_n(P)_\an$-equivariant.
\end{proof}


The previous proposition gives a morphism of pro-topological stacks $\Phi_{n,P} \colon X_\log \to \sqrt[n]{X}_\topst$. It is clear from the construction that if $n\mid m$, then the diagram
$$
\xymatrix{
X_\log\ar[r]^<<<<{\Phi_{m,P}}\ar[rd]_{\Phi_{n,P}} & \sqrt[m]{X}_\topst\ar[d]\\
 & \sqrt[n]{X}_\topst
}
$$
is (2-)commutative, so we obtain a morphism $(\Phi_X)_P\colon X_\log\to (\sqrt[\infty]{X})_{\topst}$ of pro-topological stacks.

\textbf{Step 2} (compatibility of the local constructions): let us extend this local construction to a global one. The idea is of course to use descent and glue the local constructions, and intuitively, one would expect that these local maps patch together to define a global one without incident. However, writing down all the necessary 2-categorical coherences gets pretty technical quickly, and it is much cleaner to use the machinery of $\infty$-categories.

We will need some preliminary lemmas and constructions.


\begin{lemma}\label{lemma.charts}
Let $X$ be a fine saturated log scheme over a field $k$ with two Kato charts $X\to \Spec k[P]$ and $X\to \Spec k[Q]$ for the log structure. Then for every geometric point $x$ of $X$, after passing to an \'{e}tale neighborhood of $x$, there is a third chart $X\to \Spec k[R]$ with maps of monoids $P\to R$ and $Q\to R$ inducing a commutative diagram
$$
\xymatrix{
 & &\Spec k[P] \\
X\ar[r]\ar@/^1.5pc/[rru]\ar@/_1.5pc/[rrd] &\Spec k[R]\ar[ru]\ar[rd] & \\
& & \Spec k[Q].
}
$$
\end{lemma}
\begin{proof}
We can take $R=\overline{M}_x$. There is a chart with monoid $R$ in an \'{e}tale neighborhood of $x$ by \ref{prop.chart}, and we have maps $P\to R$ and $Q\to R$ that induce a commutative diagram as in the statement, possibly after further localization.
\end{proof}



Now let us define a category $\mathfrak{I}$ of \'{e}tale open subsets of $X$ with a global chart: objects are triples $(\phi\colon U\to X,P,f)$ where $\phi\colon U\to X$ is \'{e}tale, $P$ is a fine saturated sharp monoid and $f\colon U\to \Spec \bC[P]$ is a chart for the log structure on $U$ (pulled back via $\phi$).

A morphism $(\phi\colon U\to X, P, f) \to (\psi\colon V\to X, Q, g)$ is given by a (necessarily \'{e}tale) map $U\to V$ over $X$ and a morphism $Q\to P$, such that the diagram
$$
\xymatrix{
U\ar[r]^<<<<<f\ar[d] & \Spec \bC[P]\ar[d]\\
V\ar[r]^<<<<<g & \Spec \bC[Q]
}
$$
is commutative.

We have two (lax) functors $(\blank)_\log$ and $(\sqrt[n]{\blank})_\topst\colon  \mathfrak{I}\to \Topst/X_\an$, as follows: \\for each $A=(\phi\colon U\to X, P, f) \in \mathfrak{I}$ we get, via strict pullback through the chart morphism, a local model for the Kato-Nakayama space $X_\log^A$ (over $U$) and one for the $n$-th root stack $\sqrt[n]{X}_\topst^A$. We set $A_\log=X_\log^A$ and $\sqrt[n]{A}_\topst=\sqrt[n]{X}_\topst^A$. The maps to $X_\an$ are given by the composites of the projections to $U_\an$ and the local homeomorphism $U_\an\to X_\an$. The action of these two functors on morphisms is clear.

The construction in the local case (i.e. Step 1 above) gives an assignment, for each $A\in\mathfrak{I}$, of a morphism of topological stacks $\alpha^n_A\colon A_\log\to \sqrt[n]{A}_\topst$.

\begin{lemma}
The family $(\alpha^n_A)$ gives a lax natural transformation $$\alpha^n:(\blank)_\log \Rightarrow (\sqrt[n]{\blank})_\topst,$$ in the sense of \cite[Appendix A]{hotorbi}.
\end{lemma}

\begin{proof}
By translating the definition, in the present case this means the following: if $a\colon A=(\phi\colon U\to X, P, f)\to (\psi\colon V\to X, Q, g)=B$ is a morphism in $\mathfrak{I}$, then the diagram
$$
\xymatrix{
A_\log\ar[r]^<<<<{\alpha^n_A}\ar[d] & \sqrt[n]{A}_\topst \ar[d]\ar@{}[ld]^(.30){}="a"^(.70){}="b" \ar@{=>}_{\alpha^n(a)} "a";"b"\\
B_\log\ar[r]^<<<<{\alpha^n_B} & \sqrt[n]{B}_\topst
}
$$
2-commutes, and the 2-cells $\alpha^n(a)$ satisfy a compatibility condition.

This follows from the fact that the morphism $a=(U\to V, Q\to P)$ gives a commutative diagram
$$
\xymatrix@=18px{
 & (U_n)_\log \ar[rr]\ar[ld]\ar[dd]|\hole & & (U_n)_\an \ar[ld]\\
(V_n)_\log\ar[rr]\ar[dd] & & (V_n)_\an & \\
 & X_\log^A \ar[ld] & & \\
X_\log^B & & &
}
$$
between the two objects corresponding to the functors $\alpha^n_A$ and $\alpha^n_B$. This gives a canonical natural transformation that makes the diagram
$$
\xymatrix@C=75pt{
X_\log^A\ar[d]\ar[r]^<<<<<<<<<<<<<{\alpha^n_A} & [(U_n)_\an/\mu_n(P)_\an]=\sqrt[n]{X}_\topst^A \ar[d] \ar@{}[ld]^(.40){}="a"^(.60){}="b" \ar@{=>}_{\alpha^n(a)} "a";"b"\\
X_\log^B \ar[r]^<<<<<<<<<<<<<{\alpha^n_B}&  [(V_n)_\an/\mu_n(Q)_\an]=\sqrt[n]{X}^B_\topst.
}
$$
2-commutative, and this is the required diagram.

Now if $C=(\eta\colon W\to X, R, h)$ is a third object of $\mathfrak{I}$ with a morphism $b\colon B\to C$ in $\mathfrak{I}$, then the fact that the diagram
\[
\begin{tikzcd}[row sep=small, column sep=small]
&  & (U_n)_\log \ar{rrr}\ar{ld}\ar{ddd} & & & (U_n)_\an \ar{ld}\\
& (V_n)_\log\ar[crossing over]{rrr}\ar{ddd}\ar{dl} & & & (V_n)_\an\ar{dl} & \\
(W_n)_\log\ar[crossing over]{rrr}\ar{ddd}  & & &(W_n)_\an &\\
& & X_\log^A \ar{ld} & & & \\
& X_\log^B\ar{dl} & & & &\\
X_\log^C & & & & &
\end{tikzcd}
\]
commutes implies that the composite of the two 2-cells $\alpha^n(b)$ and $\alpha^n(a)$ is equal to $\alpha^n(b\circ a)$.
\end{proof}




By composition with the natural functor

$$\Topst/X_\an \hookrightarrow \Hshi\left(\TopC\right)/X_{an}$$
to hypersheaves on $\TopC$ (see Section \ref{sec:homotopy}) and by abuse of notation we get a natural transformation of functors of $\i$-categories:

$$ \xygraph{!{0;(5,0):(0,0.18)::}
{\mathfrak{I}}="a" [r] {\Hshi\left(\TopC\right)/X_{an}.}="b"
"a":@/^{1.5pc}/"b"^-{\left(\blank\right)_{log}}|(.4){}="l"
"a":@/_{1.5pc}/"b"_-{\sqrt[n]{\blank}_{top}}
"l" [d(.3)]  [r(0.1)] :@{=>}^{\alpha^n} [d(.5)]} $$


\textbf{Step 3} (the global case): we will now use the natural transformation $\alpha^n$ above to construct a global map $$\Phi_X:X_{log} \to \sqrt[n]{X}.$$ We will first need a crucial lemma:

\begin{lemma}\label{lem:globaltrick}
Let $\iota:\mathfrak{I} \to \Hshi\left(\TopC\right)$ be the functor sending a triple $\left(\phi:U \to X,P,f\right)$ to $U_{an}.$ Then the canonical map $\colim \iota \to X_{an}$ is an equivalence.
\end{lemma}

Before proving the above lemma, we will show how we may use this lemma to produce the global morphism we seek. The key idea is the following basic fact about $\i$-topoi:

\begin{proposition}[{\emph{colimits are universal}}]
Let $\underset{i \in I} \colim A_i \to B$ be a morphism in an $\i$-topos $\cE,$ and let $C \to B$ be another morphism. Then the canonical map $$\underset{i \in I} \colim\left(C \times_B A_i\right)  \to C \times_B \underset{i \in I} \colim A_i$$ is an equivalence.
\end{proposition}

The above fact is standard and is an immediate consequence of the fact that any $\i$-topos is locally Cartesian closed.

Let us now see how we may complete the construction. Suppose we know that the canonical map $\colim \iota \to X_{an}$ is an equivalence. We can write this informally as $$\underset{\left(\phi:U \to X,P,f\right)} \colim \mspace{5mu} U_{an} \stackrel{\sim}{\longrightarrow} X_{an}.$$ Consider the morphism $X_{log} \to X_{an}.$ Then since colimits are universal we have that the following is a pullback diagram:
$$\xymatrix{\underset{\left(\phi:U \to X,P,f\right)} \colim U_{an} \times_{X_{an}} X_{log} \ar[r] \ar[d] & X_{log} \ar[d]\\
\underset{\left(\phi:U \to X,P,f\right)} \colim U_{an} \ar[r]^-{\sim} & X_{an}.}$$ It follows that top map $$\underset{\left(\phi:U \to X,P,f\right)} \colim U_{an} \times_{X_{an}} X_{log} \to X_{log}$$ is also an equivalence. However, notice that we have a canonical identification $$U_{an} \times_{X_{an}} X_{log} \cong U_{log},$$ hence $$X_{log} \simeq \underset{\left(\phi:U \to X,P,f\right)} \colim U_{log} = \colim \left(\blank\right)_{log}.$$
By a completely analogous argument, one sees that $$\sqrt[n]{X}_{top} \simeq \underset{\left(\phi:U \to X,P,f\right)} \colim \sqrt[n]{U}_{top} = \colim \sqrt[n]{\blank}_{top}.$$
For each $n$, the global map is then defined to be $$\colim \alpha^n:\colim \left(\blank\right)_{log} \to \colim \sqrt[n]{\blank}_{top}.$$ Just as in the local case, one easily sees that the maps $$\colim \alpha^n:X_{log} \to \sqrt[n]{X}_{top}$$ assemble into a morphism of pro-objects $$\Phi_X:X_{log} \to \sqrt[\i]{X}_{top}.$$ 
It is immediate from the construction that this map agrees locally with the map constructed in Step 1. In the next sections we will prove that $\Phi_X$ induces an equivalence of profinite spaces.

To finish the proof of the existence of the above map, it suffices to prove Lemma \ref{lem:globaltrick}. Without further ado, we present the proof below.

\begin{proof}[Proof of Lemma \ref{lem:globaltrick}]
Equip $\mathfrak{I}$ with the following Grothendieck topology: A collection of morphisms $$\left( \left(\phi_i:U_i \to X,P_i,f_i\right) \to \left(\phi:U \to X,P,f\right)\right)_i$$ will be a covering family if the induced family $$\left(U_i \to U\right)_i$$ is an \'etale covering family. Note that there is a canonical morphism of sites $$F: \mathfrak{I} \to X_{\et}$$ from $\mathfrak{I}$ to the small \'etale site of $X.$ Moreover, by Lemma \ref{lemma.charts}, one easily checks that $F$ satisfies the conditions of the Comparison Lemma of \cite[p. 151]{pres}, so the induced geometric morphism of topoi $$\Sh\left(\mathfrak{I}\right) \to \Sh\left(X_{\et}\right)$$ is an equivalence. It then follows
from \cite[Theorem 5]{Jardine} and \cite[Proposition 6.5.2.14]{htt} that the induced geometric morphism between the respective $\i$-topoi of hypersheaves
$$\Hshi\left(\mathfrak{I}\right) \to \Hshi\left(X_{\et}\right)$$
is an equivalence. By Remark \ref{rmk:geometric morphism}, the analytification functor is the inverse image part of a geometric morphism $$f:\Hshi\left(\TopC\right) \to \Shi\left(\Aff,\mbox{\'et}\right).$$ By \cite[Proposition 6.5.2.13]{htt}, there is an induced geometric morphism $$\tilde f:\Hshi\left(\TopC\right) \to \Hshi\left(\Aff,\mbox{\'et}\right).$$ By left Kan extension of the canonical functor $$X_{\et} \to \Hshi\left(\Aff,\mbox{\'et}\right)/X$$ which sends each \'etale open $U \to X$ to itself, one produces a colimit preserving functor $$\omega:\Hshi\left(X_{\et}\right) \to \Hshi\left(\Aff,\mbox{\'et}\right)/X.$$ Consider the composite
$$\resizebox{6.5in}{!}{$\Hshi\left(\mathfrak{I}\right) \simeq \Hshi\left(X_{\et}\right) \stackrel{\omega}{\longrightarrow} \Hshi\left(\Aff,\mbox{\'et}\right)/X \to  \Hshi\left(\Aff,\mbox{\'et}\right) \stackrel{\tilde f^*}{\longrightarrow} \Hshi\left(\TopC\right),$}$$ where $\Hshi\left(\Aff,\mbox{\'et}\right)/X \to  \Hshi\left(\Aff,\mbox{\'et}\right)$ is the canonical projection. Denote the composite by $\Theta.$ The functor $\Theta$ is colimit preserving as it is the composite of colimit preserving functors, and unwinding definitions, one sees that the composite $$\mathfrak{I} \stackrel{y}{\longrightarrow} \Hshi\left(\mathfrak{I}\right) \stackrel{\Theta}{\longrightarrow} \Hshi\left(\TopC\right)$$ is canonically equivalent to $\iota.$ It follows that there is a canonical equivalence $$\colim \iota \simeq \Theta\left(\colim y \right).$$ But $y$ is strongly generating, so by the proof of \cite[Lemma 5.3.5]{higherdave}, the colimit of $y$ is the terminal object. Unwinding the definitions, one sees that the terminal object gets sent to $X_{an}$ by $\Theta.$ This completes the proof.
\end{proof}


\section{The topology of log schemes}\label{section:topology}

This section contains preliminaries about some topological properties of fine saturated log schemes locally of finite type over $\bC$, the Kato-Nakayama space and the root stacks.

\subsection{Stratified fibrations}\label{sec:stratified}

The following proposition is a consequence of the material in Section \ref{section: Stratification}.

Recall that if $X$ is a fine saturated log scheme locally of finite typer over $\bC$ there is a stratification $\cR=\{R_n\}_{n\in\bN}$ of $X$, the \emph{rank stratification} (Definition \ref{rank.strat}), given by $R_n=\{x\in X\mid \rank_\bZ \overline{M}_{\overline{x}}^\gp\geq n \}$.

\begin{proposition}\label{strata}
The Kato-Nakayama space $X_\log$, the topological root stacks $\sqrt[m]{X}_\topst$ and the topological infinite root stack $\sqrt[\infty]{X}_\topst$ are stratified fibrations over $X_\an$ with respect to the stratification $\cR$, i.e. they are fibrations over the strata $(S_n)_\an=(R_{n}\setminus R_{n+1})_\an$ of the stratification $\cR_\an$.
\end{proposition}

\begin{proof}
All constructions are compatible with arbitrary base change along strict morphisms, so
$$ 
X_\log|_{(S_n)_\an}\cong (S_n)_\log
$$
and
$$
\sqrt[m]{X}|_{S_n}\simeq \sqrt[m]{S_n}
$$
where $m$ can be $\infty$, and $S_n$ has the log structure pulled back from $X$. It suffices then to show that the two maps $(S_n)_\log\to (S_n)_\an$ and $(\sqrt[m]{S_n})_\topst\to (S_n)_\an$ are fibrations over $S_n$.

Let us cover $(S_n)_\an$ with open subsets over which the sheaf $\overline{M}$ is constant, and recall that by definition of $S_n$ it will have rank $n$. We can choose such opens in order to have a cartesian diagram
$$
\xymatrix{
(S_n)_\log\ar[r]\ar[d] & (\Spec k[P])_\log \ar[d]\\
(S_n)_\an\ar[r] & (\Spec k[P])_\an
}
$$
over each of them, where the bottom horizontal arrow sends everything to the vertex $v_P$ (as in the proof of \ref{prop.constant}). It follows that $(S_n)_\log\cong (S^1)^n\times (S_n)_\an$, and that the map $(S_n)_\log\to (S_n)_\an$ is identified with the projection. The factor $(S^1)^n$ is the fiber of the map $ (\Spec k[P])_\log \to (\Spec k[P])_\an$ over the point $v_P$.

The analogous diagram
$$
\xymatrix{
(\sqrt[m]{S_n})_\topst \ar[r]\ar[d] & \sqrt[m]{\Spec k[P]}_\topst \ar[d]\\
(S_n)_\an\ar[r] & (\Spec k[P])_\an
}
$$
shows the same conclusion for root stacks. In this case we get an isomorphism $(\sqrt[m]{S_n})_\topst \simeq \cX \times (S_n)_\an$, where $\cX$ is the fiber of the map $ \sqrt[m]{\Spec k[P]}_\topst \to  (\Spec k[P])_\an$ over the vertex $v_P$.
\end{proof}


We need a similar (local) statement for groupoid presentations of the root stacks. 


Take $x\in X$, and an open \'{e}tale neighborhood $U\to $X of $x$ where there is a global chart $U\to \Spec \bC[P]$ for the log structure, where $P$ is fine, saturated and sharp. Then we have a quotient stack presentation for the topological $n$-th root stack $\sqrt[n]{U}_\topst \simeq (\sqrt[n]{X}|_U)_\topst $ for every $n$ (see the discussion preceding Proposition \ref{prop.charts}). 
Let us denote by $\bG(n)$ the simplicial topological space associated with this quotient presentation. There are compatible maps $\bG(m)\to \bG(n)$ whenever $n\mid m$, and the whole system gives a (simplicial) presentation for the topological infinite root stack $\sqrt[\infty]{U}_\topst$.  


Explicitly, the simplicial space $\bG(n)$ is obtained from the action of $\mu_n(P)$ on the scheme $U_n= U \times_{\Spec \bC[P]} \Spec \bC[\frac{1}{n}P]$ (see the local description of the root stacks in \ref{subsection.root}), so that $$\bG(n)_k\cong (U_n\times \mu_n(P)\times\cdots\times \mu_n(P))_\an$$ where there are $k$ copies of $\mu_n(P)$, and the map $\bG(n)_k\to U_\an$ is the composite of the projection to $(U_n)_\an$ followed by the map $(U_n)_\an\to U_\an$.


\begin{proposition}\label{simplicial.presentations}

Every $x\in U_\an$ has arbitrarily small neighborhoods over which for {every} $n$ and $k$ the map $\bG(n)_k\to U_\an$ is a product over $U_\an \cap (S_r)_\an$, where $x\in (S_r)_\an$
\end{proposition}

In particular, for every $n$ and $k$ the topological space $\bG(n)_k$ is a stratified fibration over $U_\an$.


\begin{proof}
Note first of all that since the map $U\to \Spec \bC[P]$ is strict, the rank stratification of $\Spec \bC[P]$ with its natural log structure is pulled back to the rank stratification of $U$, in the obvious sense.

Moreover from the cartesian diagram
$$
\xymatrix{
(U_n)_\an\ar[r]\ar[d] & \bC(\frac{1}{n}P) \ar[d]\\
U_\an \ar[r] & \bC(P)
}
$$
and from the fact that $\bG(n)_k\to U_\an$ is the projection $\bG(n)_k\cong (U_n\times \mu_n(P)\times\cdots\times \mu_n(P))_\an\to (U_n)_\an$ followed by $(U_n)_\an\to U_\an$, we see that it suffices to prove that the map $\pi\colon \bC(\frac{1}{n}P)=(\Spec \bC[\frac{1}{n}P])_\an\to \bC(P)=(\Spec \bC[P])_\an$ is a stratified fibration. The proof will show that for a stratum $S$ we can find an open subset $V\subseteq \bC(P)$ such that the map $\pi$ is a product over $V\cap S$ for \emph{every} $n$.


Let us pick $\phi\in \bC(P)=\Hom(P,\bC)$, and call $p_1,\hdots,p_l$ the (finitely many) indecomposable elements of $P$ (cfr.  \cite[Proposition 2.1.2]{ogus}). Assume (by reordering) that the first $h$ of those get sent to $0$ by $\phi$, and the last ones are sent to non-zero complex numbers. Call $r$ the rank of the group associated to the quotient $\sfrac{P}{\langle p_i\mid i=h+1,\hdots, l}\rangle$ (i.e. the rank of the log structure of $\bC(P)$ at $\phi$).

The stratum of the rank stratification of $\bC(P)$ to which $\phi$ belongs will be then $S_r$, the set of points of $\bC(P)$ where the log structure has rank exactly equal to $r$. It's clear that $\phi$ actually belongs to the open subset $S_\phi$ of $S_r$ consisting of the morphisms $\psi \in \Hom(P,\bC)$ such that $\psi(p_i)=0$ for $1\leq i\leq h$ and $\psi(p_i)\neq 0$ for $h< i\leq l$.

Note also that the same condition on images of indecomposables of $\frac{1}{n}P$ will determine a subset $S'_\phi\subseteq \bC(\frac{1}{n}P)=\Hom(\frac{1}{n}P,\bC)$ (of those morphisms such that the image of $\frac{p_i}{n}$ is zero exactly if $1\leq i\leq h$), that a moment's reflection will show to be exactly the preimage $\pi^{-1}(S_\phi)$. Let us check that we can choose a neighborhood of $\phi$ in $\bC(P)$ over which the restriction of $\pi\colon \pi^{-1}(S_\phi)\to S_\phi$ is a product.

For each $i=h+1,\hdots, l$ let us choose a small open disk $D_i$ around $\phi(p_i)$ in $\bC$ that does not contain the origin, and for $i=1,\hdots, h$ let $D_i$ be a small open disk around the origin. These define an open neighborhood $W$ of $\phi$ in $\bC(P)$, made up of those functions $\psi$ such that $\psi(p_i)\in D_i$ for every $i$.

Let us also choose an $n$-th root $\sqrt[n]{\phi(p_i)}$ of the non-zero complex number $\phi(p_i)$ for $i=h+1,\hdots, l$. There are a finite number of such choices, and there is a subset of those choices for which the homomorphism $\frac{1}{n}P\to \bC$ given by $\frac{p_i}{n}\mapsto \sqrt[n]{\phi(p_i)}$ is well-defined (note that this assignment might not give a well-defined homomorphism due to the relations among the indecomposable elements of the monoid $P$). Let us call $A$ this set of ``good'' choices.

Any element of $A$ determines for each $i=h+1,\hdots, l$ an $n$-th root function $\sqrt[n]{-}_i$ defined on the small disk $D_i$. Let us define a map $W\cap S_\phi \to \pi^{-1}(W\cap S_\phi)\subseteq \Hom(\frac{1}{n}P,\bC)$ by sending $\psi$ to the morphism defined by $\frac{p_i}{n}\mapsto \sqrt[n]{\psi(p_i)}_i$. This is a section of the projection $\pi^{-1}(W\cap S_\phi)\to W\cap S_\phi$, and one can check that this induces a homeomorphism $W\cap S_\phi \times A\cong \pi^{-1}(W\cap S_\phi)$, where $A$ is seen as a discrete set. We leave the details to the reader.

These arguments are uniform in $n\in \bN$, so the open subset $W$ that we identified will work for any $n$.
\end{proof}





\subsection{A system of open neighborhoods for $X_{an}$}

In this subsection we will prove the following crucial lemma:

\begin{lemma}\label{lem: open cover}
For all 
$x \in X_{an}$ there exists a fundamental system of contractible analytic open neighborhoods $\mathcal{U}_x$ of $x$ with global charts $f\colon U\to (\Spec \bC[P])_{an}$ for $U\in \mathcal{U}_x$, such that:
\begin{enumerate}
\item the map $f$ sends $x$ into the \emph{vertex} of $(\Spec \bC[P])_{an}$ (i.e. the maximal ideal generated by all non-zero elements of $P$), and  
\item each of the 
following maps is a weak homotopy equivalence:
$$
\left(X_{log}\right)_x \to X_{log}|_U
$$
and
$$\left(\mathbb{G}(n)_i\right)_x \to  \left(\mathbb{G}(n)_i\right)|_U$$
where $\{\bG(n)\}_{n\in \bN}$ is the family of topological groupoid presentations for the topological $n$-th root stack coming from the chart $f$, as in Proposition \ref{simplicial.presentations}.
\end{enumerate}
\end{lemma}


First of all we review some standard facts on triangulations and open covers.
Let $M$ be a topological space equipped with a triangulation $\cT$. Denote $\cV_\cT$ the set of vertices of $\cT$. 
If $f$ is a simplex of 
$\cT$, we denote $s(f)$ the union of the relative interiors of the 
simplices of $\cT$ that contain $f$. We call $s(f)$  the \emph{star} of $f$. 
Note that $s(f)$ is a contractible open subset of $M$. If $v$ is a vertex of $\cT$, we set $U_v\defeq s(v)$. The star of a simplex $f$ is naturally stratified by the simplices containing $f$: the strata 
are the relative interiors of the simplices containing $f$.


We say that a subspace of $\bR^n$ is a cone if it is invariant under the action 
of $\bR_{>0}$ by rescaling. We say that a cone is linear if it can be expressed as an intersection of linear spaces and linear half-spaces. 

\begin{lemma}
\label{lem:lin}
Let $v$ be in $\cV_\cT$. Then there exists an $N \in \bN$ such that $U_v$ can be embedded as a linear cone in 
$\bR^N$. Further we can choose this embedding in such a way that,  for all simplices $f$ containing $v$, 
$s(f) \subset U_v$ is mapped to a linear subcone. 
\end{lemma}
\begin{proof}
Let $v_1 \dots v_N$ be the one dimensional simplices that contain $v$, and let $e_1, \dots, e_N$ be the standard basis of $\bR^N$. If $I$ is a subset of 
$\{1, \dots, N\}$ we denote 
$$
O_I \defeq \{ \Sigma_{i \in I} \alpha_i e_i | \alpha_i \geq 0\} \subset \bR^N.
$$ 
Every simplex $\Delta$ containing $v$ determines a subset 
 $I_\Delta$ of $\{1, \dots, N\}$ in the following way: $i$ belongs to $I_\Delta$ if and only if $\Delta$ contains $v_i$. We obtain an embedding of $U_v$ into $\bR^ N$ by considering a piecewise linear homeomorphism 
 $$
 U_v \simeq \bigcup_{v \in \Delta} O_{I_\Delta}.
 $$ 
 This embedding has all the properties claimed by the lemma. 
\end{proof}

\begin{lemma}
\label{lem:prop}
Let $x$ be in $M$, and let $f$ be the lowest dimensional 
simplex such that $x$ belongs to $f$. 
Then there exists a system of open neighborhoods $\cU_x$ of $x$ such that all $U$ in $\cU_x$ have the following properties:
\begin{enumerate}
\item $U$ is contractible. 
\item $U$ does not intersect simplices of $\cT$ that do not contain $f$.
\end{enumerate}
\end{lemma}
\begin{proof}
Let $v$ be a vertex incident to $f$. By Lemma \ref{lem:lin} the open  neighborhood 
$U_v$ can be embedded as a linear cone in $\bR^n$ in such a way that $s(f) \subset U_v$ is a linear subcone. Equip $\bR^n$ with a Euclidean metric. Then $\cU_x$ can be obtained by intersecting $s(f)$ with a system of open neighborhoods given by open balls in $\bR^n$ centered at $x$. 
\end{proof}

Next we turn to the log scheme $X$. Let $x$ be in $X_{an}$. Since we are interested in constructing a system of open neighborhoods for $x$ we can assume, by \'{e}tale localizing around $x$, that 
\begin{itemize}
\item $X$ is affine, and 
\item  that we have a global chart $f\colon X\to \Spec \bC[P]$, where $P=\overline{M}_x$ (see \ref{prop.chart}), which sends $x$ to the vertex of $\Spec \bC[P]$. 
\end{itemize}
The fact that $X$ is affine is key in order to  produce triangulations, which we do in  lemma 
\ref{lem:triang}.  
By Lemma \ref{lemma.constr} the log structure determines a stratification $\cR_X$ of $X$. 

\begin{lemma}
\label{lem:triang}
There exists a triangulation $\cT_X$ of $X$ that refines $\cR_X$. 
\end{lemma}
\begin{proof}
The existence of triangulations refining stratifications of affine schemes goes back to Lojasiewicz \cite{Lo}. See also Shiota's work \cite{Sh} for a more recent reference. 
\end{proof}

By Lemma \ref{lem:prop}, the triangulation $\cT_X$ gives us  a system of open neighborhoods $\cU_x$ of $x$ in $X_{an}$ satisfying the two 
properties stated there. We claim that $\cU_x$ has all the properties required by lemma \ref{lem: open cover}. Note that, since we assumed without loss of generality that $X$ is affine and that has a global chart to $ \Spec \bC[P]$  sending $x$ to the vertex of $\Spec \bC[P]$, we only need to prove that property $(3)$ holds. We do this next.

The following lemma was proved in \cite{RSTZ}. 

\begin{lemma}[Lemma 3.25 \cite{RSTZ}]
\label{lem:rstz}
Let $W_1$ and $W_2$ be locally compact and locally contractible Hausdorff spaces. Let $p : W_1 \rightarrow W_2$ be a continuous map,
and let $K_2 \subset W_2$ be a closed deformation retract. Suppose that the restriction 
$ p^{-1}(W_2 -K_2) \rightarrow 
W_2 - K_2$ is homeomorphic to the projection from a product 
$F \times (W_2 - K_2) \rightarrow W_2 - K_2$.
Then $K_1 \defeq p^{-1}(K_2)$ is a deformation retract of $W_1$.
\end{lemma}

We will actually need a slight variant of Lemma \ref{lem:rstz}. Assume that $W_2 - K_2$ decomposes as a finite disjoint union of $m$ components, that we denote $(W_2-K_2)_i$, 
$$
W_2 - K_2 = \bigcup_{i=1}^{i=m} (W_2-K_2)_i .
$$
Then the claim still holds if the restriction $ p^{-1}(W_2 -K_2) \rightarrow 
W_2 - K_2$ is homeomorphic to the projection from a disjoint union of products 
$$
\bigcup_{i=1}^{i=m} F_i \times (W_2 - K_2)_i \rightarrow \bigcup_{i=1}^{i=m}(W_2 - K_2)_i .
$$
This stronger statement is proved exactly as Lemma \ref{lem:rstz}: and in fact, follows from it through an induction on the number of connected components of $W_2 - K_2$.

We conclude the proof of Lemma \ref{lem: open cover} by showing that the following proposition holds. 

\begin{proposition}
\label{prop:retraction-kn}
For all $U$ in $\cU_x$ each of the 
following maps is a weak homotopy equivalence: 
$$
\left(X_{log}\right)_x \to X_{log}|_U, \quad \left(\mathbb{G}(n)_i\right)_x \to  \left(\mathbb{G}(n)_i\right)|_U
$$
where $\{\bG(n)\}_{n\in \bN}$ is the family of topological groupoid presentations for the topological $n$-th root stack coming from the chart $f$.  
\end{proposition}
\begin{proof}
The proof is the same for both $X_{log}$ and $\mathbb{G}(n)_i$.  The argument relies exclusively on the fact that $X_{log}$ and $\mathbb{G}(n)_i$ give stratified fibrations on $X_{an}$ with respect to the  stratification $\cR_X$. 
To avoid repetition, we prove the statement only for $X_{log}$ but the argument remains valid if we substitute $\mathbb{G}(n)_i$ in all occurences of $X_{log}$.

Let $f$ be the lowest dimensional simplex of $\cT_X$ 
such that $x$ lies on $f$. Recall from the proof of Lemma \ref{lem:prop} that, in order to define $\cU_x$, we pick a vertex $v$ of the triangulation $\cT_X$ that is incident to $f$. By  construction, $U$ is  an open subset of 
$U_v$.  Thus $U$ carries a stratification which is obtained by restricting to it the stratification on $U_v$ by the simplices containing $v$.  

For all $k \in \bN$, denote $U_k \subset U$ the $k$-skeleton of $U$: that is, $U_k$ is the union of the strata of dimension less than or equal to $k$. Note that if $k < dim(f)$, $U_k$ is empty, and is contractible if  $dim(f) \leq k$. Further if $dim(f) \leq k \leq k'$, 
$U_k$ is a strong deformation retract of $U_{k'}$. Indeed  both 
$U_k$ and $U_{k'}$ are CW complexes (up to compactifying), and any contractible subcomplex of a contractible
CW complex is a strong deformation retract, see e.g. \cite[Lemma 1.6]{M00}.

We prove next that if $dim(f) \leq k-1$, the map
$$
X_{log}|_{U_{k-1}} \rightarrow X_{log}|_{U_k}
$$
is a deformation retract. Note that $U_k - U_{k-1}$ is equal to the disjoint union of  $k$-dimensional strata. That is $U_k - U_{k-1}$ can be 
written as a disjoint union of $m$ components
$$
U_k - U_{k-1} = \bigcup_{i=1}^{i=m} (U_k - U_{k-1})_i .
$$
The restriction of the map 
$X_{log}|_U \rightarrow X_{an}|_U$ to each stratum of $U$ is a principal bundle. Indeed, the stratification on $U$ is  finer that the restriction to $U$ of $\cR_X$. Further, it is a trivializable principal bundle, 
since the strata are paracompact Hausdorff and contractible.

Thus the restriction 
$$
X_{log}|_{U_k - U_{k-1}} \rightarrow U_k - U_{k-1}
$$
is homeomorphic to a projection from a disjoint union of products
$$
X_{log}|_{U_k - U_{k-1}} \simeq \bigcup_{i=1}^{i=m} F_i \times (U_k - U_{k-1})_i \rightarrow \bigcup_{i=1}^{i=m} (U_k - U_{k-1})_i .
$$
We have showed that the map $U_{k-1} \rightarrow U_k$ is a deformation retract. 
We apply Lemma \ref{lem:rstz}, or rather the variant that was discussed immediately after the statement of Lemma \ref{lem:rstz}, (note that $X_\log$ is locally compact Hausdorff and locally contractible by Proposition \ref{prop:hausdorff}),
and deduce that the map 
$$
X_{log}|_{U_{k-1}} \rightarrow X_{log}|_{U_k}
$$
is also a deformation retract, as we claimed.

There exists a $N \in \bN$ such that $U_N = U$. By applying recursively the retractions that we have constructed in the previous paragraph, we obtain a deformation retract
$$
X_{log}|_{U_{dim(f)}} \rightarrow X_{log}|_U .
$$
By property $(2)$ of Lemma \ref{lem:prop}, 
$U_{dim(f)}$ is connected. Further it is contractible and paracompact, and thus $X_{log}|_{U_{dim(f)}}$ is homeomorphic to a product $F \times {U_{dim(f)}}$. This implies that there are homotopy equivalences  
$$
 (X_{log})|_x \simeq F \times \{x \} \stackrel{\sim}{\ensuremath{\lhook\joinrel\relbar\joinrel\rightarrow}} F \times {U_{dim(f)}} \simeq X_{log}|_{U_{dim(f)}}
$$
and this concludes the proof. 
 \end{proof}


\section{The equivalence}\label{section: the equivalence}

Finally, in this section we will prove the main result of this paper, namely that there is an equivalence $$\Pip\left(\Phi_X\right)\colon \Pip\left(X_\log\right)\to \Pip\left(\sqrt[\infty]{X}_{\topst}\right)$$ of profinite spaces, where $\Pip$ is the ``profinite homotopy type'' functor defined in \ref{sec:profhomtypestacks}, and $\Phi_X$ is the morphism of pro-topological stacks constructed in Section \ref{section: KN to IRS}.

The main idea is to use the basis of open subsets constructed in Lemma \ref{lem: open cover} to produce a suitable hypercover of $X_{an}$ and to use this to reduce to checking that one has a profinite homotopy equivalence along fibers. First, we will need a few more technical lemmas.

The following lemma makes precise in what way one can glue profinite spaces together using hypercovers:

\begin{lemma}\label{lem:hyperprof}
Let $\cX$ be a hypersheaf in $\Hshi\left(\TopC\right).$ Let $\cI$ be a cofiltered $\i$-category and let 
$$f_{\bullet}:\cI \to \Hshi\left(\TopC\right)/\cX$$ be an $\cI$-indexed pro-system with associated pro-object $\underset{i \in \cI} \lim \left(f_i:\cY^i \to \cX\right).$
Let $$U^\bullet:\Delta^{op} \to \Hshi\left(\TopC\right)/\cX$$ be a hypercover of $\cX.$ For each $i,$ denote by $f_i^*U^\bullet$ the pullback of the hypercover $U^\bullet$ to a hypercover of $\cY^i.$ Consider the underlying pro-object $\underset{i \in \cI} \lim \cY^i$ in $\Hshi\left(\TopC\right).$ Then there is a canonical equivalence of profinite spaces $$\Pip\left(\underset{i \in \cI} \lim \cY^i\right) \simeq \underset{n \in \Delta^{op}} \colim \left[\Pip\left(\underset{i \in \cI} \lim f_i^*U^n\right)\right],$$ where
$$\widehat{\Pi}_\i:\Pro\left(\Hshi\left(\TopC\right)\right) \to \Profs$$ is the functor constructed in Section \ref{sec:profhomtypestacks}.
\end{lemma}

\begin{proof}
It suffices to show that for every $\pi$-finite space $V,$ there is a canonical equivalence $$\Hom_{\Profs}\left(\underset{n \in \Delta^{op}} \colim \left[\Pip\left(\underset{i \in \cI} \lim f_i^*U^n\right)\right],j\left(V\right)\right) \simeq \Hom_{\Profs}\left(\Pip\left(\underset{i \in \cI} \lim \cY^i\right),j\left(V\right)\right)$$ which is natural in $V.$ We have that 
$$\Hom_{\Profs}\left(\underset{n \in \Delta^{op}} \colim \left[\Pip\left(\underset{i \in \cI} \lim f_i^*U^n\right)\right],j\left(V\right)\right) \simeq \underset{n \in \Delta} \lim \left[\underset{ i \in \cI^{op}} \colim \Hom_{\cS}\left(\Pi_\i f_i^*U^n,V\right)\right].$$ Notice that $V$ is $k$-truncated for some $k$, and hence so is $j\left(V\right)$ by Proposition \ref{prop:pitruncated}. Since filtered colimits of $k$-truncated spaces are $k$-truncated, it follows that for all $n,$ $$\underset{i \in \cI^{op}}\colim \Hom_{\cS}\left(\Pi_\i f_i^*U^n,V\right)$$ is $k$-truncated. By Lemma \ref{lem:truncatedcosimplicial}, it then follows that 
$$\Hom_{\Profs}\left(\underset{n \in \Delta^{op}} \colim \left[\Pip\left(\underset{i \in \cI} \lim f_i^*U^n\right)\right],j\left(V\right)\right) \simeq \underset{n \in \Delta_{\le k}} \lim \left[\underset{ i \in \cI^{op}} \colim \Hom_{\cS}\left(\Pi_\i f_i^*U^n,V\right)\right].$$ By using that filtered colimits commute with finite limits, we then have that this is in turn equivalent to 
$$\underset{ i \in \cI^{op}} \colim \left[\underset{n \in \Delta_{\le k}} \lim  \Hom_{\cS}\left(\Pi_\i f_i^*U^n,V\right)\right].$$ Again by Lemma \ref{lem:truncatedcosimplicial} this is equivalent to $$\underset{ i \in \cI^{op}} \colim \left[\underset{n \in \Delta} \lim  \Hom_{\cS}\left(\Pi_\i f_i^*U^n,V\right)\right].$$ Finally, we have the following string of natural equivalences:
\begin{eqnarray*}
\underset{ i \in \cI^{op}} \colim \left[\underset{n \in \Delta} \lim  \Hom_{\cS}\left(\Pi_\i f_i^*U^n,V\right)\right] &\simeq& \underset{ i \in \cI^{op}} \colim \Hom_{\cS}\left(\underset{n \in \Delta^{op}} \colim \Pi_\i f_i^*U^n,V\right)\\
&\simeq& \underset{ i \in \cI^{op}} \colim \Hom_{\cS}\left(\Pi_\i \underset{n \in \Delta^{op}} \colim f_i^*U^n,V\right)\\
&\simeq& \underset{ i \in \cI^{op}} \colim \Hom_{\cS}\left(\Pi_\i\cY^i,V\right)\\
&\simeq& \Hom_{\Profs}\left(\Pip\left(\underset{i \in \cI} \lim \cY^i\right),j\left(V\right)\right).
\end{eqnarray*}
\end{proof}

Let $X$ be a log scheme. Denote by $\cU$ the basis of contractible open subsets of $X_{an}$ given by Lemma \ref{lem: open cover}. 

\begin{lemma}
There is a hypercover $$U^\bullet:\Delta^{op} \to \TopC/X_{an}$$ such that for all $n,$ the map $U^n \to X_{an}$ is isomorphic to the coproduct of inclusions of open neighborhoods in the basis $\cU,$ and all the structure maps are local homeomorphisms.
\end{lemma}

\begin{proof}
Using standard techniques, since $\cU$ is a basis for the topology of $X_{an}$ we can construct a split hypercover satisfying the above by induction (cfr. \cite{conrad}).
\end{proof}

\begin{remark}
The image under the Yoneda embedding of the hypercover of topological spaces $U^\bullet$ just constructed is a hypercover of $Y\left(X_{an}\right)$ in the $\i$-topos $\Hshi\left(\TopC\right).$ We will abuse notation by identifying the two.
\end{remark}

We now prove our main result:

\begin{theorem}\label{thm:main}
Let $X$ be a fine saturated log scheme locally of finite type over $\bC$. The induced map $$\Pip\left(\Phi_X\right)\colon \Pip\left(X_{\log}\right)  \stackrel{\sim}{\longlongrightarrow} \Pip\left(\sqrt[\infty]{X}_{top} \right)$$ is an equivalence of profinite spaces.
\end{theorem}


\begin{proof}
Consider now the hypercover $U^\bullet$ of $X_{an}$ just constructed. Then each $U^n=\underset{\alpha} \coprod V_{\alpha}$ where each $V_\alpha$ is in $\cU.$ Let us restrict to one such $V=V_\alpha.$ Since $V$ is in $\cU,$ there exists an $x \in V$ such that $\left(X_{log}\right)_x \to X_{log}|_V$ is a weak homotopy equivalence, and such that there is a Kato chart $U \to \Spec \bC[P]$, with $U \to X$ \'etale, such that $U_{an} \to X_{an}$ admits a section $\sigma$ over $V,$ and with the property that the composite
$$V \stackrel{\sigma}{\longrightarrow} U_{an} \to (\Spec \bC[P])_{an}$$ carries $x$ to the vertex point of the toric variety $\Spec \bC[P].$ Let us fix this $x$, and call it the \emph{center} of $V$. Suppose that the monoid $P$ has rank $k,$ then the log structure at $x$ also has rank $k$. Moreover, the fiber of the map $$V_n\defeq V \times_{(\Spec \bC[P])_{an}} \left(\Spec \bC\left[\frac{1}{n}P\right]\right)_{an} \to V$$ over $x$ consists of a single point (see \cite[Lemma 1.2]{illusie-nakayama-tsuji}). 

Let us fix an $n$, then we have that $$\sqrt[n]{X}_{top}|_{V}\simeq \left[\left(\mathbb{Z}/n\mathbb{Z}\right)^k \ltimes V_n\right]=\left[V_n/\left(\mathbb{Z}/n\mathbb{Z}\right)^k\right].$$ Hence our groupoid presentation $\mathbb{G}\left(n\right)$ for $\sqrt[n]{X}_{top}|_{V}$ guaranteed by Proposition \ref{simplicial.presentations} is the topological action groupoid $\left(\mathbb{Z}/n\mathbb{Z}\right)^k \ltimes V_n.$ This groupoid admits a continuous functor to $V$ (viewing $V$ as a topological groupoid with only identity arrows) which on objects is simply the canonical map $V_n \to V.$ Similarly, regard the one-point space $*$ also as a topological groupoid, and consider the canonical map $$* \to V$$ picking out $x.$ Since $V$ and $*$ have no non-identity arrows, the lax fibered product of topological groupoids $$* \times^{\left(2,1\right)}_{V} \left(\left(\mathbb{Z}/n\mathbb{Z}\right)^k \ltimes V_n\right)$$ is equivalent to the strict fibered product $$* \times_{V} \left(\left(\mathbb{Z}/n\mathbb{Z}\right)^k \ltimes V_n\right)$$ which is canonically equivalent to the action groupoid $$\left(\mathbb{Z}/n\mathbb{Z}\right)^k \ltimes \left(V_n\right)_x,$$ where $\left(V_n\right)_x$ is the fiber over $V_n \to V.$ Since this fiber consists of a single point, we conclude that the lax fibered product may be identified with $\left(\mathbb{Z}/n\mathbb{Z}\right)^k,$ where we are identifying the group $\left(\mathbb{Z}/n\mathbb{Z}\right)^k$ with its associated $1$-object groupoid. 

Consider the continuous functor of topological groupoids 
$$\left(\mathbb{Z}/n\mathbb{Z}\right)^k\simeq * \times^{\left(2,1\right)}_{V} \left(\left(\mathbb{Z}/n\mathbb{Z}\right)^k \ltimes V_n\right) \to \left(\mathbb{Z}/n\mathbb{Z}\right)^k \ltimes V_n.$$ This induces a map of simplicial topological spaces between their simplicially enriched nerves
$$N\left(\left(\mathbb{Z}/n\mathbb{Z}\right)^k\right) \to N\left(\left(\mathbb{Z}/n\mathbb{Z}\right)^k \ltimes V_n\right).$$ By Lemma \ref{lem: open cover}, this map is degree-wise a weak homotopy equivalence. It follows from Proposition \ref{prop:fat} and \cite[Lemma 3.2]{hafdave} that the induced map 
$$B\left(\left(\mathbb{Z}/n\mathbb{Z}\right)^k\right) \simeq \Pi_\i\left(\left(\mathbb{Z}/n\mathbb{Z}\right)^k \ltimes *\right) \to \Pi_\i\left(\left[\left(\mathbb{Z}/n\mathbb{Z}\right)^k \ltimes V\right]\right) \simeq \Pi_\i\left(\sqrt[n]{X}_{top}|_{V}\right)$$
is an equivalence in $\cS.$ Since the topological groupoid presentations for $\sqrt[n]{X}_{top}$ constructed in Section \ref{sec:stratified} are compatible with the natural maps $\sqrt[m]{X}_{top} \to \sqrt[n]{X}_{top}$ when $n\mid m$, it follows that we have a natural identification of $$\Pip\left(\sqrt[\i]{X}_{top}|_{V}\right)\simeq \underset{n} \lim B\left(\left(\mathbb{Z}/n\mathbb{Z}\right)^k\right)$$ in $\Profs.$ Consider the pro-system of finite groups $$n \mapsto \left(\mathbb{Z}/n\mathbb{Z}\right)^k.$$ This is the $k^{th}$ Cartesian power of the pro-system $$n \mapsto \left(\mathbb{Z}/n\mathbb{Z}\right),$$ which is simply $\widehat{\mathbb{Z}}.$ By Proposition \ref{prop:infinity-stone}, it follows that $$\Pip\left(\sqrt[\i]{X}_{top}|_{V}\right)\simeq B\left(\widehat{\mathbb{Z}}^k\right),$$ and hence by Proposition \ref{prop:profiniteemergency}, we have that $$\Pip\left(\sqrt[\i]{X}_{top}|_{V}\right) \simeq B\left(\widehat{\mathbb{Z}^k}\right).$$
We also have that $$\left(X_{log}\right)_x \cong \left(S^1\right)^k.$$ It follows that $$\Pi_\i \left(X_{log}|_V\right) \simeq \Pi_\i\left(S^1\right)^k\simeq B\left(\mathbb{Z}^k\right),$$ and so $$\Pip\left(X_{log}|_V\right) \simeq \widehat{B\left(\mathbb{Z}^k\right)}.$$ Since $\mathbb{Z}^k$ is a finitely generated free abelian group, it is \emph{good} in the sense of Serre in \cite{Serre}. It follows from \cite[Proposition 3.6]{Qu2} and Theorem \ref{thm:prohom} that the canonical map $$\widehat{B\left(\mathbb{Z}^k\right)} \to B\left(\widehat{\mathbb{Z}^k}\right)$$ is an equivalence of profinite spaces, hence $$\Pip\left(X_{log}|_V\right) \simeq B\left(\widehat{\mathbb{Z}^k}\right).$$
It now follows that $$\Pip\left(\sqrt[\i]{X}_{top}|_{V}\right)\simeq \Pip\left(X_{log}|_V\right),$$which is a local version of our statement.

Now let us globalize using the hypercover $U^\bullet$. For each $n$, denote by $q_n$ the natural map $$q_n:\sqrt[n]{X}_{top} \to X_{an}.$$ Since $\Pip$ preserves colimits, it follows that the induced map $$\underset{l \in \Delta^{op}}\colim \Pip\circ \tau^*U^l \to \underset{l \in \Delta^{op}}\colim \left( \Pip\circ \underset{n} \lim q_n^*U^l\right)$$ is an equivalence of profinite spaces, where $\tau$ is the canonical map $\tau:X_{log} \to X_{an}.$ However, $$\underset{l \in \Delta^{op}}\colim \Pip\circ \tau^*U^l \simeq \Pip\left(\underset{l \in \Delta^{op}}\colim  \tau^*U^l\right)\simeq \Pip\left(X_{log}\right),$$ since $\tau^*U^\bullet$ is a hypercover of $X_{log}.$ Finally, by Lemma \ref{lem:hyperprof}, $$ \underset{l \in \Delta^{op}}\colim \left( \Pip\circ \underset{n} \lim q_n^*U^l\right)\simeq \Pip\left( \underset{n } \lim \sqrt[n]{X}_\topst\right)=\Pip\left(\sqrt[\i]{X}_{top}\right).$$\end{proof}



\section{The profinite homotopy type of a log scheme}\label{section:prohottype}
We conclude this paper by defining the profinite homotopy type of an arbitrary log scheme over a ground ring $k,$ by using the notion of \'etale homotopy type. 

\'Etale homotopy theory, as originally introduced by Artin and Mazur in \cite{ArtinMazur}, is a way of associating to a suitably nice scheme a pro-homotopy type. In this seminal work they proved a generalized Riemann existence theorem:

\begin{theorem}[{\cite[Theorem 12.9]{ArtinMazur}}]
Let $X$ be scheme of finite type over $\mathbb{C},$ then the profinite completion of the \'etale homotopy type of $X$ agrees with the profinite completion of $X_{an}$.
\end{theorem}

In light of the above theorem, the \'etale homotopy type of a complex scheme of finite type  gives a way of accessing homotopical information about its analytic topology by using only algebro-geometric information, and for a setting where the analytic topology is not available, such as a scheme over an arbitrary base, the profinite completion of its \'etale homotopy type serves as a suitable replacement.

In the original work of Artin and Mazur, for $X$ a locally Noetherian scheme, one associates a pro-object in the homotopy category of spaces $\mbox{Ho}\left(\cS\right).$ This definition was later refined by Friedlander in \cite{Friedlander} to produce a pro-object in the category $\Set^{\Delta^{op}}$ of simplicial sets, and a generalized Riemann existence theorem is also proven in this context. In recent work of Lurie \cite{dagxiii}, the \'etale homotopy type of an arbitrary higher Deligne-Mumford stack is defined by using \emph{shape theory} to produce an object in the $\i$-category $\Pro\left(\cS\right)$ (in fact the definition in op. cit. is for \emph{spectral} Deligne-Mumford stacks - analogues of Deligne-Mumford stacks for algebraic geometry over $\mathbb{E}_\infty$-rings), and Hoyois has recently proven that up to profinite completion, this definition agrees with that of Friedlander for a classical locally Noetherian scheme in \cite{Hoyois}. See also recent work of Barnea, Harpaz, and Horel in \cite{prohomotopy}.

In recent work of the first author \cite{etalehomotopy}, the \'etale homotopy type of an arbitrary higher stack on the \'etale site of affine $k$-schemes is defined, and is shown to agree with the definition of Lurie when restricted to higher Deligne-Mumford stacks. In particular, there is shown to be a functor
$$\widehat{\Pi}^{\et}_\i:\Shi\left(\Affk,\mbox{\'et}\right) \to \Profs$$ associating to a higher stack $\cX$ on the \'etale site of affine $k$-schemes of finite type a profinite space $\widehat{\Pi}^{\et}_\i\left(\cX\right)$ called its \emph{profinite homotopy type}, and an even more generalized Riemann existence theorem is proven:

\begin{theorem}\label{3hot}\cite[Theorem 4.13]{etalehomotopy} 
Let $\cX$ be higher stack on affine schemes of finite type over $\mathbb{C},$ then there is a canonical equivalence of profinite spaces $$\widehat{\Pi}^{\et}_\i\left(\cX\right)\simeq \Pip\left(\cX_{top}\right)$$ between its profinite \'etale homotopy type and the profinite homotopy type of its underlying topological stack $\cX_{top}$ in the sense of Theorem \ref{thm:analification}.
\end{theorem}

Now let $X$ be a log scheme locally of finite type over $\bC.$ Its infinite root stack $\sqrt[\i]{X}$ is a pro-object in $\Shi\left(\Aff,\mbox{\'et}\right).$ Notice that the functor $\widehat{\Pi}^{\et}_\i$ canonically extends to a functor
$$\widehat{\Pi}^{\et}_\i:\Pro\left(\Shi\left(\Affk,\mbox{\'et}\right)\right) \to \Profs.$$ In light of the above theorem, we conclude that there is a canonical equivalence of profinite spaces $$\widehat{\Pi}^{\et}_\i\left(\sqrt[\i]{X}\right) \simeq \Pip\left(\sqrt[\i]{X}_{top}\right)$$ between the profinite \'etale homotopy type of the infinite root stack $\sqrt[\i]{X}$ and the profinite homotopy type of the underlying topological stack of the infinite root stack $\sqrt[\i]{X}_{top}.$ Combining this with Theorem \ref{thm:main} yields the following theorem.

\begin{theorem}
Let $X$ be a log scheme locally of finite type over $\bC.$ Then the following three profinite spaces are canonically equivalent:
\begin{itemize}
\item[i)] The profinite completion $\widehat{X_{log}}$ of its Kato-Nakayama space.
\item[ii)] The profinite homotopy type $\Pip\left(\sqrt[\i]{X}_{top}\right)$ of the underlying topological stack of its infinite root stack $\sqrt[\i]{X}.$
\item[iii)] The profinite \'etale homotopy type $\widehat{\Pi}^{\et}_\i\left(\sqrt[\i]{X}\right)$ of its infinite root stack $\sqrt[\i]{X}.$
\end{itemize}
\end{theorem}

In light of the above theorem, we make the following definition:

\begin{definition}
Let $X$ be a log scheme over a ground ring $k.$ Then the \textbf{profinite homotopy type} of $X$ is the profinite \'etale homotopy type of its infinite root stack $\sqrt[\i]{X}.$
\end{definition}


\appendix

\section{}\label{preliminaries}

In this appendix we gather some definitions and results about log schemes, analytification, the Kato-Nakayama space, root stacks and topological stacks. 

\subsection{Log schemes}

Log (short for ``logarithmic'') schemes were first defined and studied systematically in \cite{kato}. A modern introduction (with a view towards moduli theory) can be found in \cite{abramovich}.

\begin{remark}
We will give definitions and facts in the algebraic category, but we will apply them to the complex-analytic context as well. The only difference is that instead of the \'{e}tale topology we will be using the analytic topology.
\end{remark}

\begin{definition}
A \textbf{log scheme} is a scheme $X$ with a sheaf of monoids $M$ on the small \'{e}tale site $X_{\et}$ and a homomorphism $\alpha\colon M\to \cO_X$ of sheaves of monoids, where $\cO_X$ is seen as a monoid with respect to multiplication of regular functions, such that $\alpha$ induces an isomorphism
$$
\alpha|_{\alpha^{-1}(\cO^\times_X)}\colon \alpha^{-1}(\cO^\times_X)\to \cO^\times_X.
$$ 
\end{definition}

Note that the last condition gives us a canonical embedding $\cO_X^\times\hookrightarrow M$ as a subsheaf of groups.

We denote a log scheme by $(X,M,\alpha)$ or sometimes simply by $X$. 


\begin{example}\mbox{ }
\begin{itemize}
\item Any scheme $X$ is a log scheme with $M=\cO_X^\times$ and $\alpha$ the inclusion. This is the trivial log structure on $X$.
\item Any effective Cartier divisor $D\subseteq X$ induces a log structure, by taking $M$ to be the subsheaf of $\cO_X$ given by functions that are invertible outside of $D$.
\item If $P$ is a monoid, the spectrum of the monoid algebra $X_P\defeq \Spec k[P]$ has a natural log structure. The sheaf $M$ is obtained by considering the natural map $P\to k[P]=\Gamma(\cO_{X_P})$ and taking the ``associated log structure'' (see below for a few more details).
\end{itemize}
\end{example}

Log structures can be pulled back and pushed forward along morphisms of schemes. In particular

\begin{itemize}
\item any open subscheme of a log scheme can be equipped with the restriction of the log structure,
\item if we have a morphism of schemes $f\colon X\to \Spec k[P]$ we get an induced log structure on $X$. This happens in the following way: $f$ gives a morphism of monoids $P\to \cO_X(X)$, that induces $\widetilde{\alpha}\colon\underline{P}\to \cO_X$ where $\underline{P}$ is the constant sheaf. It is typically not true that $\widetilde{\alpha}$ induces an isomorphism between $\widetilde{\alpha}^{-1}\cO^\times_X$ and $\cO^\times_X$, but there is a procedure to fix the behaviour of the units, and this produces a log structure $\alpha\colon M\to \cO_X$. See \cite[Example 1.5]{kato} for details.
\end{itemize}

\begin{remark}\label{remark.stalks}
We remark that in the situation of the last bullet, the quotient $M/\cO_X^\times$ is obtained from $\underline{P}$ by locally ``killing the sections of $\underline{P}$ that become invertible in $\cO_X$'', so in particular all the stalks of $M/\cO_X^\times$ are quotients of the monoid $P$.
\end{remark}

We consider only coherent log structures, which are those that, \'{e}tale locally, come by pullback from the spectrum of the monoid algebra of a monoid.

\begin{definition}
A log scheme $X$ is \textbf{quasi-coherent} if there is an \'{e}tale cover $U_i$ of $X$, monoids $P_i$ and morphisms of log schemes $f_i\colon U_i\to \Spec k[P_i]$ that are strict, i.e. the log structure on $U_i$ is pulled back from $\Spec k[P_i]$ via $f_i$. The monoid $P_i$ and the map $f_i$ are a \textbf{chart} for the log structure over $U_i$.

A log scheme $X$ is \textbf{coherent} (resp. \textbf{fine}, resp. \textbf{fine and saturated}) if the monoids $P_i$ above can be taken finitely generated (resp. finitely generated and integral, resp. finitely generated, integral and saturated).
\end{definition}

A morphism such as $f_i$ in the definition above that identifies the pullback of the log structure on the target with the one of the source will be called \textbf{strict}.


We are interested only in fine and saturated log schemes.

\begin{proposition}[{{\cite[Proposition 2.1]{Ols}}}]\label{prop.chart}
Let $X$ be a fine saturated log scheme and $x$ a geometric point. Then there exists an \'{e}tale neighborhood $U$ of $x$ over which there is a chart for the log structure with monoid $P=(M/\cO_X^\times)_x$.
\end{proposition}

This says in particular that if $X$ is fine and saturated, we can locally find charts with $P$ fine, saturated and sharp.

The quotient sheaf $\overline{M}=M/\cO^\times_X$ is called the characteristic sheaf of the log structure. Taking the quotient (in an appropriate sense) by $\cO^\times_X$ of the map $\alpha$, we get an alternative definition of a (quasi-integral) log scheme, introduced in \cite{borne-vistoli}.

Let us denote by $\Div_X$ the fibered category over $X_{\et}$ whose objects over $U\to X$ are pairs $(L,s)$ where $L$ is an invertible sheaf of $\cO_U$-modules on $U$ and $s$ is a global section. This is a symmetric monoidal fibered category, where the monoidal operation is given by tensor product.

\begin{definition}
A \textbf{log scheme} is a scheme $X$ together with a sheaf of monoids $A$ and a symmetric monoidal functor $L\colon A\to \Div_X$ with trivial kernel.
\end{definition}

The phrasing ``trivial kernel'' in the definition means that if a section $a$ is such that $L(a)$ is isomorphic to $(\cO_X,1)$ in $\Div_X$, then $a=0$.

Given a (quasi-integral) log scheme $(X,M,\alpha)$, by taking the ``stacky quotient'' of $\alpha\colon M\to \cO_X$ by $\cO^\times_X$ we get the functor $L\colon A=\overline{M}=[M/\cO^\times_X]\to [\cO_X/\cO_X^\times]=\Div_X$. Quasi-integrality ensures that the quotient $[M/\cO^\times_X]$ is actually a sheaf. Of course integral log structures are quasi-integral. See \cite[Theorem 3.6]{borne-vistoli} for details.

One can give a notion of charts in this context as well. For many purposes these two notions of chart can be used indifferently. We mostly use charts as in the first definition above. These are called ``Kato charts'' in \cite{borne-vistoli}.

\begin{remark}
A fist approximation of how one should ``visualize'' a log scheme is by thinking about the stalks of the sheaf $\overline{M}$. This sheaf is locally constant on a stratification of $X$ (see Proposition \ref{prop.constant}) and the stalks are fine saturated sharp monoids. Of course this disregards the particular extension $M$ of $\overline{M}$ by $\cO_X^\times$ and the map $\alpha$ (or equivalently the functor $L$), so it is indeed just a crude approximation.
\end{remark}

\subsection{Analytification}

We are mainly concerned with log schemes locally of finite type over $\bC$, and with their analytifications.

Recall that if $X$ is a scheme locally of finite typer over $\bC$, the associated analytic space $X_\an$ is defined as a set as the $\bC$ points $X(\bC)=X(\Spec\bC)$ of $X$. This has an ``analytic'' topology coming from the local embeddings into $\bC^n$. Moreover this construction extends also to algebraic spaces locally of finite type over $\bC$ (see \cite{artin, toen-vaquie}).

If $X$ is a log scheme locally of finite type over $\bC$, the analytication $X_\an$ inherits a log structure, because of the relationship between the \'{e}tale topos of $X$ and the analytic topos of $X_\an$.  An \'{e}tale morphism $X\to Y$ induces a local homeomorphism $X_\an\to Y_\an$, that consequently has local sections in the analytic topology. This gives a functor from the \'{e}tale site of $X$ to the analytic site of $X_\an$, and induces a morphism of topoi. The log structure on $X_\an$ is obtained via this functor. Thus, 
in what follows every time something holds \'{e}tale locally for the log scheme $X$, it will hold analytically locally for the log analytic space $X_\an$.

We will use this without further mention, and we will use the same letter to denote the sheaf of monoids $M$ on $X$ and the induced one on $X_\an$. This should cause no real confusion.

\begin{definition}\label{def.monoid}
For a monoid $P$ we denote by $\bC(P)$ the analytification of the spectrum of the monoid algebra $\Spec \bC[P]$.
\end{definition}

As sets we have $\bC(P)=\Hom(P,\bC)$, the set of homomorphisms of monoids, where $\bC$ is given the multiplicative structure.

A basis of opens of $\bC(P)$ (where $P$ is fine, saturated and sharp) can be described as follows: call $p_1,\hdots, p_k$ the indecomposable elements of $P$ (see \cite[Proposition 2.1.2]{ogus}), and choose open disks $D_i$ in the complex plane $\bC$. Then the set of homomorphisms $\phi\in \Hom(P,\bC)$ such that $\phi(p_i)\in D_i$ is open in $\bC(P)$. 
Letting the disks $D_i$ vary we get a basis for the open subsets of $\bC(P)$.

\begin{lemma}[{\cite[p. 12]{toen-vaquie}}]
Analytification commutes with finite limits.
\end{lemma}

We will need the following result on the topological properties of analytifications of schemes locally of finite type over $\bC$. As references, we point out \cite{Lo}.

\begin{proposition}
Let $X$ be an affine scheme of finite type over $\bC$ and $Y\subseteq X$ be a closed subscheme. Then there exist compatible triangulations of $X_\an$ and $Y_\an$, realizing $Y_\an$ as a subcomplex.
\end{proposition}

We can apply this iteratively to a stratification, to get compatible triangulations of the ambient affine scheme and of all the (closed) strata.

\subsection{Kato-Nakayama space}

From now on all log schemes will be fine and saturated unless we specify otherwise. Just for this subsection, $X$ will denote an analytic space rather than a scheme.

The Kato-Nakayama space $X_\log$ of a log analytic space $X$ (for example of the form $Y_\an$ for some log scheme $Y$ locally of finite type over $\bC$) is a topological space introduced in \cite{KN}. 
The idea is to define a topological space that ``embodies'' the log structure of $X$ in a topological way (i.e. without using the sheaf of monoids, but only ``points'').

What comes out is a topological space $X_\log$ (that also comes with a natural sheaf of rings, but we do not use this in the present work) with a continuous map $\tau\colon X_\log\to X$ that is proper and surjective. Moreover if $U\subseteq X$ is the trivial locus of the log structure (the largest open subset over which $\cO_X^\times\hookrightarrow M$ is an isomorphism), the open embedding $i\colon U\to X$ factors through $\tau$, so that $X_\log$ can be considered as a `` relative compactification'' of the open immersion $i$.

Let us denote by $p^\dagger$ the log analytic space given by the point $\mathrm{pt}=(\Spec \bC)_\an$ with monoid $M=\bR_{\geq 0}\times S^1$, and map $\alpha\colon M\to \bC$ described by $(r,a)\mapsto r\cdot a$. Note that this log structure is not integral.

As a set we have $X_\log=\Hom(p^\dagger,X)$, the set of morphisms of log analytic spaces from the log point $p^\dagger$ to $X$. By unraveling this one can also write
$$X_\log=\left\{(x,c) \mid  c \colon M^\gp_x \to S^1 \mbox{ is a group hom. such that } c(f)=\frac{f}{|f|} \mbox{ for all } f\in \cO_{X,x}^\times\right\}.$$

In particular one can see that $\bC(P)_\log=\Hom(p^\dagger,\bC(P))=\Hom(P,\bR_{\geq 0}\times S^1)$, and the projection $\tau\colon \bC(P)_\log\to \bC(P)$ is given by post-composition with $\bR_{\geq 0}\times S^1\to \bC$.

Note that from the above description $\bC(P)_\log$ has a natural topology, that by means of local charts for the log structure gives a topology on $X_\log$ in general \cite[Section 1.2]{KN}.

From the description one sees easily that for $x\in X_\an$, the fiber $\tau^{-1}(x)$ is homeomorphic to $(S^1)^r$ where $r$ is the rank of the stalk $\overline{M}_x$, defined to be the rank of the free abelian group $\overline{M}_x^\gp$.

The construction of the Kato-Nakayama space is clearly functorial, and is also compatible with strict base change. 


\begin{proposition}[{{\cite[Lemma 1.3]{KN}}}]\label{kn.basechange}
Let $f\colon X\to Y$ be a strict morphism of fine saturated log analytic spaces. Then the diagram of topological spaces
$$
\xymatrix{
X_\log\ar[r] \ar[d]& Y_\log\ar[d]\\
X\ar[r] & Y
}
$$
is cartesian.
\end{proposition}



The description of $X_\log$ as a set can actually be enhanced to a description of its functor of points (see Section \ref{subsec.KN}). 




We now prove that following proposition.

\begin{proposition}\label{prop:hausdorff}
For any log scheme $X,$ the Kato-Nakayama space $X_{log}$ is locally Hausdorff, locally contractible and locally compact. 
\end{proposition}

We will start by assuming that $X$ is affine and has a global chart $X\to \Spec \bC[P]$ for a fine saturated sharp monoid $P$, and will prove that $X_\log$ is locally compact, Hausdorff and locally contractible. 
This implies the conclusion for arbitrary $X$.

Note that since $f\colon X\to \Spec \bC[P]$ is strict, there is a Cartesian diagram of topological spaces
$$
\xymatrix{
X_\log\ar[r]\ar[d] & \bC(P)_\log\ar[d]^\tau\\
X_\an\ar[r]^{f_\an} & \bC(P).
}
$$
Our proof will be as follows: we note that $X_\an$ and $\bC(P)$ are semialgebraic, and the map $X_\an\to \bC(P)$ is a semialgebraic function (this part of the diagram is even algebraic). We will check that $\bC(P)_\log$ is semialgebraic, and that the projection to $\bC(P)$ is a semialgebraic function.

After we do that, it will follow that $X_\log$ is semialgebraic as well (being the inverse image of the diagonal $\bC(P)\subseteq \bC(P)\times \bC(P)$, a semialgebraic set, through a semialgebraic map $(f_\an,\tau) \colon X_\an\times \bC(P)_\log\to \bC(P)\times\bC(P)$, see  \cite[Proposition 2.2.7]{real}), hence triangulable (by the results of \cite{Lo}), and any triangulable locally semialgebraic set is locally compact, Hausdorff and locally contractible \cite{hofmann}.



\begin{lemma}
The space $\bC(P)_\log$ is semialgebraic, and the projection $\bC(P)_\log \to \bC(P)$ is a semialgebraic map.
\end{lemma}

\begin{proof}
We will check this by writing out these spaces explicitly. Let $p_i$ be a finite set of generators for $P$ (for example the indecomposable elements), and assume to have a finite number of relations that present the monoid $P$, of the form $\sum_j r_{ij} p_j=\sum_j s_{ij}p_j$. Say there's $k$ generators and $h$ relations.

Then we have a map $\bC(P)=\Hom(P,\bC) \to \bC^k$ given by $\phi\mapsto (\phi(p_i))$. This is an embedding, and the closed image is the Zariski closed subset with equations $\prod_j (z_j)^{r_{ij}}=\prod_j (z_j)^{s_{ij}}$ obtained from the $h$ relations of the chosen presentation of $P$, and $(z_j)$ are the coordinates of $\bC^k$.

In the exact same way we have a map $\bC(P)_\log=\Hom(P,\bR_{\geq  0}\times S^1)\to (\bR_{\geq  0}\times S^1)^k$ given by $\psi\mapsto (\psi(p_i))$. To describe the image, let us note that we have $\bR_{\geq 0}\times S^1\subseteq \bR^3$ in a natural way, as a semialgebraic subset. If we denote by $(\zeta_j)$ the ``coordinates'' of $(\bR_{\geq  0}\times S^1)^k$, then the (isomorphic) image of $\bC(P)_\log$ is again described by the equations $\prod_j (\zeta_j)^{r_{ij}}=\prod_j (\zeta_j)^{s_{ij}}$, so it is semialgebraic (the equations translate into algebraic equations on $(\bR^3)^k$).

Of course the diagram
$$
\xymatrix{
\bC(P)_\log \ar[r]\ar[d] & (\bR_{\geq  0}\times S^1)^k \ar[d]\\
\bC(P)  \ar[r] & \bC^k
}
$$
commutes.

From this, it suffices to check that the map $(\bR_{\geq  0}\times S^1)^k\to \bC^k$ is semialgebraic, and this is easy: in coordinates (where we see $(\bR_{\geq  0}\times S^1)^k\subseteq (\bR^3)^k$ and $\bC^k\cong (\bR^2)^k$) it is given by $(a_i,b_i,c_i)\mapsto (a_i\cdot b_i,a_i\cdot c_i)$.
\end{proof}

\subsection{Root stacks}\label{subsection.root}

Root stacks of log schemes were introduced in \cite{borne-vistoli}. The infinite root stack, an inverse limit of the ones with finitely generated  weight system, is the subject of \cite{TV}. We briefly recall the functorial definition and the groupoid presentations coming from local charts.

Let us fix a natural number $n$ and a log scheme $X$ with log structure $L\colon A\to \Div_X$. We can consider a sheaf $\frac{1}{n}A$ of ``fractions'' of sections of $A$: the sections of $\frac{1}{n}A$ are formal fractions $\frac{a}{n}$ where $a$ is a section of $A$. There is a natural inclusion $i_n\colon A\to \frac{1}{n}A$.

Note that $\frac{1}{n}A$ is isomorphic to $A$ via $a\mapsto \frac{a}{n}$. Through this isomorphism, the inclusion $i_n$ corresponds to multiplication by $n \colon A\to A$. The fact that this map is injective follows from torsion-freeness of stalks of $A$, which are fine saturated sharp monoids.

\begin{definition}
The \textbf{$n$-th root stack} $\sqrt[n]{X}$ of the log scheme $X$ is the stack over $\Sch,$ the category of schemes (with the \'{e}tale topology), whose functor of points sends a scheme $T$ to the groupoid whose objects are pairs $(\phi, N, a)$ where $\phi\colon T\to X$ is a morphism of schemes, $N\colon \frac{1}{n}\phi^*A\to \Div_X$ is a symmetric monoidal functor with trivial kernel and $a$ is a natural isomorphism between $\phi^*L$ and the composite $N\circ i_n$.
$$
\xymatrix@=.5cm@R=.35cm{ \phi^*A\ar[rr]\ar[dd] & \ar@{=>}[d]!<-3ex,0ex>^a& \Div_X\\ & & \\ \frac{1}{n}\phi^*A\ar[uurr] & & }
$$
Morphisms are defined in the obvious way.
\end{definition}

In other words the $n$-th root stack parametrizes extensions of the symmetric monoidal functor $L\colon A\to \Div_X$ to the sheaf $\frac{1}{n}A$. The pair $(N,a)$ in the definition above could be called an ``$n$-th root'' of the log structure $L\colon A\to \Div_X$.

Every time $n\mid m$ there is a morphism $\sqrt[m]{X}\to \sqrt[n]{X}$, and by letting $n$ and $m$ vary, these maps give an inverse system of stacks over $\Sch$.

\begin{definition}
The \textbf{infinite root stack} $\sqrt[\infty]{X}$ of the log scheme $X$ is the pro-algebraic stack $(\sqrt[n]{X})_{n\in\bN}$.
\end{definition}

\begin{remark}
In \cite{TV} the infinite root stack is defined as the actual limit of the inverse system in the $2$-category of fibered categories, but in the present paper it will always be the pro-object. We remark that the two contain the same information, since by the results of \cite[Section 5]{TV} the limit of the system of $n$-th root stacks recovers the log scheme completely, and hence recovers the pro-object as well.
\end{remark}



The $n$-th root stack $\sqrt[n]{X}$ is a tame Artin stack with coarse moduli space $X$. 
Moreover there are presentations of $\sqrt[n]{X}$ for each $n$ that assemble into a pro-object in groupoids in schemes, and can be regarded as a presentation of the pro-object $\sqrt[\infty]{X}$. This follows from the following local descriptions as quotient stacks \cite[Corollary 3.12]{TV}.

Let us fix a monoid $P$, and let us denote by $C_n$ the cokernel of the injective map $P^\gp\to \frac{1}{n}P^\gp$. Furthermore denote by $\mu_n(P)$ the Cartier dual of $C_n$. This acts on the monoid algebra $\Spec k[\frac{1}{n}P]$ ($k$ here is some base field, but this works the same way over $\bZ$).

If $X$ is a log scheme with a global chart $X\to \Spec k[P]$, then there is a cartesian diagram
$$
\xymatrix{
\sqrt[n]{X}\ar[r]\ar[d] & [\Spec k[\frac{1}{n}P]/\mu_n(P)]\ar[d]\\
X \ar[r] & \Spec k[P]
}
$$
presenting $\sqrt[n]{X}$ as a quotient stack $[X_n/\mu_n(P)]$, where $X_n=X\times_{\Spec k[P]} \Spec k[\frac{1}{n}P]$.


As we mentioned, these quotient stack presentations are all compatible, in the sense that they give a pro-object in groupoids in schemes $(X_n\times \mu_n(P)\rightrightarrows X_n)_{n\in\bN}$, that can be seen as a groupoid presentation of $\sqrt[\infty]{X}$. 

If $X$ does not have a global chart we cover it with \'{e}tale opens $U_i$ where there is a chart with monoid $P_i$ 
and assemble together the corresponding groupoid presentations.




\begin{proposition}[{{\cite[Proposition 4.19]{borne-vistoli}}}]\label{prop.charts}
The $n$-th root stack $\sqrt[n]{X}$ is a tame Artin stack, and is Deligne--Mumford when we are over a field of characteristic $0$.

\end{proposition}

\subsection{Topological stacks}

The main reference for this section is \cite{No1}.

The two preceding subsections were about the objects that we would like to compare, namely the Kato-Nakayama space and the infinite root stack of a log scheme locally of finite type over $\bC$. Note that the former is of topological nature, and the latter is algebraic. In order to find a map between them, we carry over the root stacks to the topological side.

One can talk about stacks over any Grothendieck site. Algebraic stacks (a.k.a. Artin stacks) are stacks on the category of schemes over a base with the \'{e}tale topology\footnote{Sometimes, rather than working with the \'etale topology, one defines algebraic stacks with the \emph{fppf} topology. However, the resulting $2$-category of stacks is the same, cfr. \cite[\href{http://stacks.math.columbia.edu/tag/076U}{Tag 076U}]{stacks-project}.} that admit a representable smooth epimorphism from a scheme and whose diagonal is representable by algebraic spaces (and often one imposes some conditions on this diagonal morphism, like being quasi-compact or locally of finite type). Equivalently, one can describe algebraic stacks as stacks of (\'etale) torsors for certain groupoid objects in algebraic spaces, whose structure maps are smooth.

If instead of schemes over a base with the \'{e}tale topology we start from topological spaces with the \'{e}tale topology (where covers are local homeomorphisms), and we require a representable epimorphism from a topological space, we obtain the theory of topological stacks\footnote{In \cite{No1}, Noohi demands further conditions for such a stack to be called a topological stack, however in subsequent papers (e.g. \cite{No2}), he relaxes these conditions to the ones just described.}. Such a stack will always have diagonal representable by a topological space. As on the algebraic side, a topological stack can be defined through a groupoid presentation: a topological stack is a stack of principal $\cG$-bundles for $\cG$ a topological groupoid, and much of the basic yoga that one learns when working with algebraic stacks carries over in close analogy in this context.

In particular if $G$ is a topological group acting on a space $X$, the functor of points of the quotient stack $[X/G]$ is described as principal $G$-bundles (the topological analogue of $G$-torsors) with an equivariant map to $X$. In the same fashion, if $R\rightrightarrows U$ is a topological groupoid, one can characterize the associated stack $[U/R]$ as the stack of principal bundles for this groupoid.

There is a procedure to produce a topological stack starting from an algebraic one, that extends the analytication functor. We apply this in particular to the $n$-th root stacks of a log scheme.

Denote by $\Algst$ the $2$-category of algebraic stack locally of finite type over $\bC$ and by $\Topst$ the $2$-category of topological stacks.

\begin{proposition}[{{\cite[Section 20]{No1}}}]\label{noohi}
There is a functor of $2$-categories $$\left(\blank\right)_{top}:\mathbf{A}\!\mathfrak{lgSt}^{LFT}_{\mathbb{C}} \to \mathfrak{TopSt}$$ that associates a topological stack to an algebraic stack locally of finite type over $\mathbb{C}.$
\end{proposition}

In Section \ref{sec:homotopy}, we extend Noohi's results to produce a left exact colimit preserving functor from $\i$-sheaves (a.k.a. stacks of $\i$-groupoids) on the algebraic \'etale site, to hypersheaves on a suitable topological site. See Theorem \ref{thm:analification} and Corollary \ref{cor:algstacktotopstack}.

 This functor has several nice properties. We point out the ones that we use:
 \begin{itemize}
 \item[1.] If $X$ is a scheme (or algebraic space) locally of finite type over $\bC$, then $X_\topst\simeq X_\an$ is the analytification
\item[2.] The functor $\left(\blank\right)_{top}$ preserves all finite limits (i.e. is left exact).
\item[3.] The preceding properties give us a procedure for calculating $\cX_{top}$ for an algebraic stack $\cX$. If $R\rightrightarrows U$ is a groupoid presentation of $\cX$ where $R$ and $U$ are locally of finite type and the maps are smooth, then by the first property we can apply the analytification functor to the diagram, and by the second one, this will result in another groupoid, namely the groupoid in topological spaces $R_\an\rightrightarrows U_\an$. The topological stack $\cX_\topst$ is then the associated stack $[U_\an/R_\an]$.
\end{itemize}

In particular if $\cX=[U/G]$ for an action of an algebraic group locally of finite type $G$ on a scheme locally of finite type $X$, we have $\cX_\topst=[U_\an/G_\an]$.

\begin{definition}
Let $X$ be a log scheme locally of finite type over $\bC$. The \textbf{topological $n$-th root stack} of $X$ is the topological stack $\sqrt[n]{X}_\topst$.
As for the algebraic ones, the topological root stacks form an inverse system. The pro-topological stack $\sqrt[\infty]{X}_\topst\defeq (\sqrt[n]{X}_\topst)_{n\in\bN}$ is the \textbf{topological infinite root stack} of $X$.
\end{definition}






\subsection{The rank stratification}\label{section: Stratification}



In this section we will prove that the characteristic sheaf $\overline{M}$ is locally constant on a stratification over the log scheme $X$. This is used in the main body of this article to prove that the Kato-Nakayama space and the infinite root stack are ``stratified fibrations'' over $X$, and that the map that we construct between them induces an equivalence of profinite completions.

The results of this part are probably known to experts, and we are including them because of the lack of a suitable reference.

\begin{definition}
By a \textbf{stratification} of a topological space $T$ we mean a collection of closed subsets $\cS=\{S_i\subseteq T\}_{i\in I}$ where $I$ is partially ordered, and the following are satisfied:
\begin{itemize}
\item if $i\leq j$, then $S_i\subseteq S_j$, and
\item the stratification is locally finite: every point $t\in T$ has an open neighborhood $U$ such that only finitely many of the intersections $U\cap S_i$ are non-empty.
\end{itemize}
The locally closed subsets $S_j\setminus S_i$ will be called the \emph{strata} of the stratification.
\end{definition}

If in the above definition $T$ is the underlying topological space of a scheme $X$ and each $S_i$ is Zariski closed, we will say that $\cS$ is an \emph{algebraic} stratification of the scheme $X$. Note that an algebraic stratification on $X$ will induce a stratification on the analytification $X_\an$.



\begin{definition}\label{strat.fibr}
Let $T$ be a topological space equipped with a stratification $\cS$, and let $f\colon T'\to T$ be a morphism, where $T'$ is a topological space or stack. We will say that $f$ is a \textbf{stratified fibration} with respect to $\cS$ if the restrictions of $f$ to the strata of $\cS$ are fibrations (in our case, this will always mean ``locally the projection from a product'').
\end{definition}

Now let $X$ be a log scheme locally of finite type over a field $k$. We will describe an algebraic stratification of $X$ over which the sheaf $\overline{M}$ is locally constant. 

The basic idea is that we are stratifying by the rank of the stalks $\overline{M}_x^\gp$ of the sheaf of abelian groups $\overline{M}^\gp$.

\begin{lemma}[{{\cite[3.5]{Ols}}}]\label{lemma.constr}
The sheaf  $\overline{M}^\gp$ is a constructible sheaf of $\bZ$-modules \cite[IX 2.3]{SGA4}. This means that (Zariski locally) there is a decomposition of $X$ into locally closed subsets over which $\overline{M}^\gp$ is a locally constant sheaf.
\end{lemma}

\begin{lemma}[{{\cite[Theorem 2.3.2]{ogus}}}]\label{lemma.closed}
If $\xi$ is a generalization of $\eta$ in $X$, meaning that $\eta \in \overline{\{\xi\}}$, then there is a natural morphism of the stalks $\overline{M}_{\overline{\eta}}\to \overline{M}_{\overline{\xi}}$, and this is surjective (more specifically, it is a quotient by a \emph{face}).
\end{lemma}

This last lemma follows from Proposition \ref{prop.chart} and from the explicit description of the stalks of the monoid $\overline{M}$ of the log structure obtained from a chart, see Remark \ref{remark.stalks}.

In particular the rank ``only jumps up in closed subsets'', i.e. for every $n\in\bN$ the subset $R_n$ of points of $X$ where the rank of the group $\overline{M}_{\overline{x}}^\gp$ is $\geq n$ is closed: it is constructible by Lemma \ref{lemma.constr}, and stable under specialization by Lemma \ref{lemma.closed}, so it is closed. Note also that $R_{n+1}\subseteq R_n$.

\begin{definition}\label{rank.strat}
The \emph{rank stratification} of a log scheme $X$ is the algebraic stratification $\cR=\{R_n\}_{n\in\bN}$, where
$$
R_n=\{x\in X\mid \rank_\bZ \overline{M}_{\overline{x}}^\gp\geq n \}.
$$
We will denote the strata by $S_n \defeq R_n\setminus R_{n+1}$.
\end{definition}

For example, $R_0=X$ and the complement $X\setminus R_1$ is the open subset of $X$ where the log structure is trivial (which might be empty). In general $S_n$ is the locally closed subset of $X$ over which the rank of $\overline{M}_{\overline{x}}^\gp$ is equal to $n$.

We claim that both sheaves $\overline{M}$ and $\overline{M}^\gp$ are locally constant on the strata $S_n$.

To check this, let's describe the canonical log structure $\overline{M}_P \to \Div_{X_P}$ on $X_P=\Spec k[P]$ in more detail: the log structure is induced by the morphism of monoids $P\to k[P]$, which gives a morphism of sheaves of monoids $\underline{P} \to \cO_{X_P}$ (here $\underline{P}$ denotes the constant sheaf), from which we get the sheaf $\overline{M}_P$ by killing the preimage of the units in $\cO_{X_P}$. More precisely, denote by $\{p_i\}_{i\in I}$ the finitely many indecomposable elements of the fine saturated monoid $P$; these are generators of $P$. For a geometric point $x \to X_P$ call $S\subseteq I$ the subset of indices such that the image of $t^{p_i}\in k[P]$ is invertible in the residue field $k({x})$. Then the the stalk $(\overline{M}_P)_{x}$ is the quotient $\sfrac{P}{\langle p_i\mid i \in S\rangle}$.

In particular we note the following:

\begin{lemma}\label{lemma.rank}
The only point $x$ of $X_P$ where the stalk $(\overline{M}_P)_{{\overline{x}}}$ has rank $n=\rank_\bZ P^\gp$ is the ``vertex'' $v_P$, the point given by the maximal ideal $\langle t^{p_i}\mid i \in I \rangle$ generated by the variables corresponding to the indecomposable elements of $P$. 
\end{lemma}

The point $v_P$ is also sometimes referred to as the ``torus-fixed point''.

\begin{proof}
Since $P^\gp\cong \bZ^n$ for some $n$, as soon as at least one of the indecomposable elements $p_i$ is killed, the rank will drop at least by $1$. The only point in which no indecomposable is killed is exactly the maximal ideal generated by all the $t^{p_i}$.
\end{proof}

\begin{proposition}\label{prop.constant}
For every $n$ and every point $x$ of $S_n=R_n\setminus R_{n-1}$ there is an \'{e}tale neighborhood $U\to S_n$ of $x$ such that the sheaves $\overline{M}|_{S_n}$ and $\overline{M}^\gp|_{S_n}$ are constant sheaves.
\end{proposition}

\begin{proof}
If we equip $R_n$ with the reduced subscheme structure, it is a (fine saturated) log scheme with the log structure pulled back from $X$, and the same is true for the open subset $S_n\subseteq R_n$. Consequently there is an \'{e}tale neighborhood $U\to S_n$  of $x$ and a chart $U\to \Spec k[P]$ for the induced log structure on $U$, where $P=\overline{M}_{\overline{x}}$ (Proposition \ref{prop.chart}). If $\overline{M}_P$ is the sheaf of monoids for the canonical log structure on $\Spec k[P]$, there is exactly one point where the stalk has rank $n=\rank \; P$ (=$\rank_\bZ P^\gp$), corresponding to the vertex $v_P$ (Lemma \ref{lemma.rank}).

This implies (since over $U$ the rank of the stalks of $\overline{M}$ is always $n$) that the morphism $U\to \Spec k[P]$ sends everything to $v_P$, and in turn that the sheaf $\overline{M}|_U$, being a pullback from $\Spec k[P]$, is constant. This implies that $\overline{M}^\gp|_U$ is constant as well, and concludes the proof.
\end{proof}

Note that if $k=\bC$, the algebraic stratification of $X$ we just constructed induces a stratification of the analytification $X_\an$, and the sheaves $\overline{M}$ and $\overline{M}^\gp$ of the log analytic space are locally constant over the strata.


\bibliography{profinite}
\bibliographystyle{hplain}

\end{document}